\documentclass[a4paper,oneside]{article}
\usepackage{amsmath, amssymb, amsthm,graphicx}
\usepackage{xcolor,color}
\usepackage[font=small,labelfont=md,textfont=it]{caption}
\usepackage{floatrow,float}
\usepackage[titletoc, title]{appendix}
\usepackage[colorlinks,linkcolor=blue,citecolor=blue]{hyperref}
\usepackage{longtable}
\usepackage{pgfplots}
\usepackage{diagbox}
\usepackage{booktabs,makecell,multirow}
\usepackage[capitalise, nosort]{cleveref}
\usepackage{algorithm}
\usepackage{algorithmic}
\usepackage{amsmath,bm}
\usepackage{cases}
\usepackage{geometry}
\usepackage{extarrows}
\usepackage{tcolorbox}
\usepackage{bbm}
\usepackage{syntonly}
\usepackage[figuresright]{rotating}

\crefname{equation}{}{}
\crefname{lem}{Lemma}{Lemmas}
\crefname{thm}{Theorem}{Theorems}

\newcommand{\dd}{\,{\rm d}}
\newcommand{\R}{\,{\mathbb R}}

\newcommand{\dual}[1]{\left\langle {#1} \right\rangle}

\newcommand{\proxi}[0]{ {\bf prox}}

\newcommand{\dom}[0]{ {\bf dom\,}}
\newcommand{\argmin}[0]{ {\mathop{{\rm  argmin}}\,}}

\newcommand{\nm}[1]{\left\lVert {#1} \right\rVert}
\newcommand{\snm}[1]{\left\lvert {#1} \right\rvert}
\newcommand{\ssnm}[1]
{
	\left\vert\kern-0.25ex
	\left\vert\kern-0.25ex
	\left\vert
	{#1}
	\right\vert\kern-0.25ex
	\right\vert\kern-0.25ex
	\right\vert
}

\makeatletter
\def\spher@harm#1{%
	\vbox{\hbox{%
			\offinterlineskip
			\valign{&\hb@xt@2\p@{\hss$##$\hss}\vskip.2ex\cr#1\crcr}%
		}\vskip-.36ex}%
}
\def\gshone{\spher@harm{.}}
\def\gshtwo{\spher@harm{.&.}}
\def\gshthree{\spher@harm{.&.&.}}
\let\gsh\spher@harm
\makeatother

\newtheorem{coro}{Corollary}[section]
\newtheorem{prop}{Proposition}[section]
\newtheorem{Def}{Definition}[section]

\newtheorem{lem}{Lemma}[section]

\newtheorem{rem}{Remark}[section]

\newtheorem{thm}{Theorem}[section]

\newcounter{mnote}
\let\oldmarginpar\marginpar
\renewcommand\marginpar[1]
{\-\oldmarginpar[\raggedleft\footnotesize #1]{\raggedright\footnotesize #1}}

\makeatletter\def\@captype{table}\makeatother

\newfloatcommand{capbtabbox}{table}[][\FBwidth]
\usepackage[T1]{fontenc} 
\usepackage[utf8]{inputenc} 
\usepackage{authblk}

\begin{document}
	\title{
		\Large \bf Accelerated 
		differential inclusion for convex optimization
	}
	\author{Hao Luo\thanks{School of Mathematical Sciences, 
			Peking University, Beijing, 100871, China (Email: luohao@math.pku.edu.cn).}
	}
	\date{} 
	\maketitle
	\begin{abstract}
This paper introduces a second-order differential inclusion for unconstrained convex optimization. In continuous level, solution existence in proper sense is obtained
and exponential decay of a novel Lyapunov function along with the solution trajectory 
is derived as well. Then in discrete level, based on numerical discretizations of the continuous model, two inexact proximal point algorithms are proposed, and some new convergence rates are established via a discrete Lyapunov function.
	\end{abstract}
	
	\section{Introduction}
	In this work, we investigate the accelerated differential inclusion
	\begin{equation}\label{eq:ADI-intro}
		\gamma x'' + (\mu+\gamma)x'+\partial f(x)\ni \xi,
	\end{equation}
	where $f:\mathcal H\to\R\cup\{+\infty\}$ is a proper and closed convex function on the finite-dimensional Hilbert space $\mathcal H$ and 
	$\xi\in C([0,\infty);\mathcal H)$ stands for small perturbation. The time scaling factor $\gamma$ in \cref{eq:ADI-intro} is governed by $\gamma' = \mu-\gamma$, and $\mu\geqslant 0$ stands for the strongly convex constant (cf. \cref{eq:mu}) of $f$. Throughout, $\mathcal H$ is equipped with the inner product $\dual{\cdot,\cdot}$ and the norm $\nm{\cdot} = \sqrt{\dual{\cdot,\cdot}}$.
	Besides, assume $\argmin f\neq\emptyset$ and let $x^*\in\argmin f$ be a global minimizer of $f$.
	
	Recently, many authors investigate first-order optimization methods from the 
	ordinary differential equation point of view. The gradient flow 
	\cite{lemaire_asymptotical_1996} 
	models the gradient descent method 
	and the proximal point algorithm (PPA) \cite{rockafellar_monotone_1976}. 
	The heavy ball system \cite{balti_asymptotic_2016,haraux_second_2012} 
	is associated with Polyak's heavy ball method \cite{polyak_methods_1964}. 
	The asymptotic vanishing damping (AVD) model, which is derived in \cite{
		Su;Boyd;Candes:2016differential} 
	and further studied and generalized in \cite{attouch_rate_2019,
		wibisono_variational_2016},
	recoveries Nesterov accelerated gradient (NAG) method \cite{Nesterov1983} 
	and FISTA \cite{beck_fast_2009}. The dynamical systems considered
	in \cite{attouch2020,chen_luo_first_2019,luo_chen_differential_2019,Siegel:2019,Wilson:2018} are closely related to Nesterov's optimal method \cite[Chapter 2]{Nesterov:2013Introductory} 
	based on estimate sequence technique. When $\xi=0$, the differential inclusion \cref{eq:ADI-intro} reduces to the NAG
	flow proposed in our previous work \cite{luo_chen_differential_2019}.
	
	Investigations of these dynamical systems with perturbation have also
	been studied by many authors. 
	In \cite{haraux_second_2012,jendoubi_asymptotics_2015}, the authors considered 
	the following heavy ball system with perturbation 
	\begin{equation}\label{eq:hb-xi}
		x''(t) + \alpha(t)x'(t)+\nabla f(x(t)) = \xi(t), 
		\quad t>0,
	\end{equation}
	where $\alpha(\cdot)$ is some continuous positive function on $\R_+$. 
	However, only weak convergence result 
	of the trajectory to \cref{eq:hb-xi} was established. 
	The perturbed generalization of the AVD  model \cite{Su;Boyd;Candes:2016differential} reads as follows
	\begin{equation}\label{eq:3/t-xi}
		x''(t) + \frac{\alpha}{t^\beta} x'(t) +\nabla f(x(t)) = \xi(t), 
		\quad t\geqslant t_0>0,
	\end{equation}
	where $\alpha>0$ and $0\leqslant \beta\leqslant 1$. 
	For $\beta=  1$, under the assumption $\xi\in L^1_{\nu}(0,\infty;\mathcal H)$ with $\nu(t) = (t+1)^{\min\{1,\alpha/3\}}$,
	the decay rate $O(t^{-\min\{3,2\alpha/3\}})$ 
	can be found in \cite{attouch_fast_convergence_2018,attouch_rate_2019,aujol_optimal_2017}. Here and in the sequel, $L^1_\nu(0,\infty;\mathcal H)$ denotes the standard $\mathcal H$-valued weighted $L^1$ space, with $\nu$ being some nonnegative measurable function on $(0,\infty)$. Balti and May \cite{balti_asymptotic_2016} studied the case $0\leqslant \beta<1$ 
	and established the rate $O(t^{-2\beta})$,
	provided that 
	$\xi\in L^1_{\nu}(0,\infty;\mathcal H)$ with $\nu(t) = (t+1)^{\beta}$. 
	Those decay estimates  yield the minimizing property of the solution trajectory $x(t) $, which converges weakly (or strongly under further assumption) to a limit in $\argmin f$.
	Based on this observation, Sebbouh et al. \cite{sebbouh_nesterovs_2019} 
	analysed the convergence behaviour of \cref{eq:3/t-xi}
	when $f$ satisfies some further local geometrical assumptions
	such as the flatness condition and the \L ojasiewicz property, 
	which are beyond the convexity of $f$. 
	They proved that there exists $T\geqslant 0$
	such that 
	\begin{equation}\label{eq:conv-ode-Sebbouhh}
		f(x(t))-f(x^*)+ \nm{x(t)-x^*}^2 + \nm{x'(t)}^2
		\leqslant Ce^{-mt^{1-\beta}},\quad t\geqslant T,
	\end{equation}
	with the condition $\xi\in L^1_{\nu}(0,\infty;\mathcal H)$,
	where $\nu(t) =e^{mt^{1-\beta}}$, $m\in(0,2\theta/(2+\theta))$ with some 
	$\theta\in[1,2]$ related to the local geometrical property of $f$. Particularly, if $f$ is strongly convex, then $\theta=1$ and $T=0$. Thus, \cref{eq:conv-ode-Sebbouhh} implies global exponential decay rate $O(e^{-2t/3})$, requiring that $\xi\in L^1_{\nu}(0,\infty;\mathcal H)$ with $\nu(t) =e^{2t/3}$. However, for this case, our accelerated differential inclusion \cref{eq:ADI-intro} can achieve the same decay rate under weaker assumption on $\xi$; see \cref{thm:conv-Lt-mu>0}.  
	
	For the nonsmooth setting, it is worth noticing some works related to the corresponding differential inclusions. In this case, the gradient flow becomes a first-order differential inclusion \cite{attouch_variational_2014}
	\begin{equation}\label{eq:1st-DI}
		x'(t)+\partial f(x(t))\ni\xi(t),\quad t>0.
	\end{equation}
	Since $\partial f$ is a set-valued maximal monotone operator, discontinuity can occur in $x'$ and classical $C^1$ solution to \cref{eq:1st-DI} may not exist. 
	For the second-order differential inclusion
	\begin{equation}\label{eq:2nd-DI}
		x''(t)+\partial f(x(t))\ni \xi\big(t,x(t),x'(t)\big),\quad t>0,
	\end{equation}
	the concept of energy-conserving solution has been introduced by \cite{paoli_existence_2000,schatzman_class_1978}. 
	Recently, Vassilis et al. \cite{vassilis_differential_2018} 
	extended the AVD model \cite{Su;Boyd;Candes:2016differential} to the nonsmooth setting:
	\[
	x''(t)+\frac{\alpha}{t}x'(t)+\partial f(x(t))\ni 0,\quad 
	\quad t\geqslant t_0>0,
	\]
	which is a special case of \cref{eq:2nd-DI}. 
	They obtained the existence of energy-conserving solution and established the minimizing property. Note that our model \cref{eq:ADI-intro} can also be viewed as a particular case of \cref{eq:2nd-DI}, and therefore the existence of energy-conserving solution can be established. Moreover, we shall derive new exponential decay estimate
	\[
	f(x(t))-f(x^*)
	\leqslant 
	e^{-t}\left(2\mathcal L_0+
	\nm{\xi}^2_{L^1_{\nu}(0,t;\mathcal H)}\right)\quad \text{ for all } t>0,
	\]
	provided that $\xi\in L^1_{\nu}(0,\infty;\mathcal H)$ with $\nu(t) = \sqrt{e^t/\gamma(t)}$. 
	With weaker condition on $\xi$, e.g. $\nm{\xi(t)} = O(t^{-p})$ with $p>0$, the rate $O(t^{-2p})$ can be obtained if $f$ is strongly convex assumption; see \cref{rem:rate-mu>0}.
	
	We now turn to algorithm aspect. 
	Numerical discretizations for perturbed 
	dynamical systems naturally lead to inexact optimization solvers such as inexact proximal gradient method (PGM), inexact PPA and the corresponding accelerated variants.
	On the other hand, convergence analyses of inexact convex optimization methods have already been widely studied. In summary, almost all the algorithms consider three types of approximations for the proximal mapping (see their definitions in \cref{sec:inexact-proxi}). 
	
	In the pioneering work \cite{rockafellar_monotone_1976}, 
	Rockafellar has considered an inexact PPA using
	$type$-3 approximation,
	and the convergence result was derived with 
	summable error. Then G\"{u}ler \cite{guler_new_1992} 
	proposed an  inexact accelerated PPA
	involving $type$-3 approximation as well
	and derived the rate $O(1/k^2)$ with computation error
	$\varepsilon_k=O(1/k^{3/2})$. In \cite{salzo_inexact_2012}, 
	Salzo and Villa studied the convergence of 
	an inexact PPA with $type$-1
	approximation and proved the rate $O(1/k)$ with error 
	$\varepsilon_k=O(1/k^2)$. However, in \cref{sec:AIPPA}, based on an implicit discretization of \cref{eq:ADI-intro}, we shall propose an inexact accelerated
	PPA using $type$-1 approximation 
	and prove the improved rate $O(1/k^2) $ with $\varepsilon_k=O(1/k^2)$.
	
	Villa et al. \cite{villa_accelerated_2013} 
	extended the convergence result in \cite{salzo_inexact_2012} to an inexact accelerated PGM. Aujol and Dossal 
	\cite{aujol_stability_2015} investigated the 
	inertial forward-backward algorithm with $type$-1 and 
	$type$-2 approximations, and carefully studied the convergence 
	rates under decay assumption on the error. Schmidt et al.
	\cite{schmidt_convergence_2011} considered an accelerated PGM  involving  $type$-1 approximation 
	and inexact gradient data. But no detailed convergence rate is 
	given with specific error, saying $\varepsilon_{k}=O(1/k^p)$. 
	In \cref{sec:AIPGM}, from a semi-implicit discretization of \cref{eq:ADI-intro}, 
	we obtain an accelerated PGM which adopts $type$-1 approximation and inexact gradient data as well. We shall analyse both convex and strongly convex cases and present some new estimates. Particularly, for strongly convex case, we establish the rate $O(1/k^{2p})$ when the computation error decays like $O(1/k^p)$. 
	
	To the end, we introduce some conventional functional spaces and list the arrangement of the rest part. Recall that $\nm{\cdot}$ is the underlying norm of the Hilbert space $\mathcal H$; we also use $\nm{\cdot}_{X}$ to denote the norm for any normed space $X$. Given $-\infty<a<b\leqslant \infty$, let $M(a,b;\mathcal H)$ be the space of $\mathcal H$-valued Radon measures; for $k\in\mathbb N$, $C^{k}(a,b;\mathcal H)$ stands for the space of $\mathcal H$-valued functions that are $k$-times continuous differentiable and set $C(a,b;\mathcal H) = C^0(a,b;\mathcal H)$; for $1\leqslant p\leqslant \infty$, $L^p(a,b;\mathcal H)$ is the conventional $\mathcal H$-valued $L^p$ space and for $1\leqslant  p<\infty$, denote by $L^p_\nu(a,b;\mathcal H)$ the standard $\mathcal H$-valued weighted $L^p$ space, where $\nu$ is some nonnegative measurable function on $(a,b)$; for $k\in\mathbb N$, let $W^{k,\infty}(a,b;\mathcal H)$ be an $\mathcal H$-valued Sobolev space \cite{attouch_variational_2014}; the space of all $\mathcal H$-valued functions with
	bounded variation is defined by $BV(a,b;\mathcal H)$ \cite{attouch_variational_2014}. For $\mu\geqslant 0$, let $\mathcal S_\mu^0$ be 
	the set of all properly closed convex functions on $\mathcal H$ such that
	\begin{equation}\label{eq:mu}
		f(y)\geqslant f(x)+\dual{p,y-x}+\frac{\mu}{2}\nm{x-y}^2,
	\end{equation}
	for all $x,y\in\dom f$ and $p\in\partial f(x)$. For $0\leqslant \mu\leqslant L<\infty$, we say $f\in\mathcal S_{\mu,L}^{1,1}$ if it belongs to $\mathcal S_{\mu}^0$ and has $L$-Lipschitz continuous gradient, i.e.,
	\[
	\nm{\nabla f(x)-\nabla f(y)}\leqslant L\nm{x-y}\quad \text{ for all } x,y\in \mathcal H.
	\]
	
	The rest of this paper is organized as follows. 
	In \cref{sec:exist}, existence of the energy-conserving
	solution is established, and 
	minimizing property of the solution trajectory 
	is proved as well. 
	Then in \cref{sec:AIPPA,sec:AIPGM}, an inexact accelerated PPA and an inexact accelerated PGM are obtained, respectively, from the implicit and semi-implicit discretizations, and convergence rate estimates are derived by using the tool of discrete Lyapunov function. Finally, in \cref{sec:num}, to investigate the performance of the proposed methods, two numerical experiments are presented.
	\section{Existence and Minimizing Property}
	\label{sec:exist}
	In this section, assume $f\in\mathcal S_\mu^0$ with $
	\mu\geqslant 0$ and the interior of the domain of $f$, denoted 
	by ${\rm int }\,\dom f$, is nonempty. We shall focus on the existence 
	of energy-conserving 
	solution (cf. \cref{def:solu}) to the accelerated differential inclusion
	\begin{equation}\label{eq:ADI}
		\gamma x'' + (\mu+\gamma)x'+\partial f(x)\ni \xi,
	\end{equation}
	with initial data
	\begin{equation}\label{eq:x0-x1}
		x (0) = x_0\in\dom f\quad\text{and}\quad x'(0) = x_1\in 
		\overline{\mathop{\cup}\limits_{\tau>0}\tau(x_0-\overline{\dom f})}.
	\end{equation}
	The scaling factor $\gamma$ satisfies 
	\begin{equation}\label{eq:gamma}
		\gamma' = \mu-\gamma,\quad \gamma(0) = \gamma_0>0,
	\end{equation}
	from which it is easy to obtain the exact solution $\gamma(t) = \mu+(\gamma_0-\mu)e^{-t}$.
	Hence $\gamma$ is positive and bounded, i.e., $\gamma_{\min}:=\min\{\gamma_0,\mu\}\leqslant \gamma(t)\leqslant \max\{\gamma_0,\mu\}=:\gamma_{\max}$ for all $t \geqslant 0$, and $\gamma(t)\to\mu$ as $t\to\infty$.
	\subsection{The Moreau--Yosida regularization}
	\label{sec:MY}
	Given any $\lambda >0$, introduce the {\it Moreau--Yosida 
		regularization} of $f$ by that \cite{rockafellar_convex_1970} 
	\begin{equation}\label{eq:MY-f}
		f_\lambda(x) := \inf_{y\in \mathcal H}\left(f(y)+\frac{1}{2\lambda}\nm{y-x}^2\right)
		\qquad \forall\,x\in \mathcal H,
	\end{equation}
	where the infimum is attained at the unique minimizer
	\begin{equation}\label{eq:prox-f}
		\proxi_{\lambda f}(x) :=\mathop{ \argmin}\limits_{y\in \mathcal H}\left(f(y)+\frac{1}{2\lambda}\nm{y-x}^2\right)
		\qquad \forall\,x\in \mathcal H.
	\end{equation}
	Therefore, one must have
	\begin{equation}\label{eq:sub-inf}
		\frac{1}{\lambda}\left(x-\proxi_{\lambda f}(x)\right)
		\in\partial f(\proxi_{\lambda f}(x)),
	\end{equation}
	and it follows that any fixed-point of $\proxi_{\lambda f}$ belongs 
	to $\argmin f$. On the other hand, by definition, it is not hard to see (cf. \cite[Remark 12.24]{Bauschke2019})
	\begin{equation}\label{eq:f_lambda}
		f_{\lambda}(x) =f(\proxi_{\lambda f}(x)) +  
		\frac{1}{2\lambda}\nm{x-\proxi_{\lambda f}(x)}^2,
	\end{equation}
	and it holds $f(\proxi_{\lambda f}(x)) \leqslant 	f_{\lambda}(x) 	\leqslant f(x)$ for all $x\in \mathcal H$.
	Hence, any $x^*\in\argmin f$ satisfies $x^* = \proxi_{\lambda f}(x^*)$, and we conclude that $f_\lambda$ has the same minimizer and minimum as that of $f$. Besides, by \cite[Proposition 17.2.2]{attouch_variational_2014}, we have the convergence property $\lim\limits_{\lambda\to 0}f_{\lambda}(x) = f(x)$ for all $x\in \mathcal H$.
	
	By \cite[Proposition 17.2.1]{attouch_variational_2014}, we have the identity $\nabla f_\lambda(x) = 
	\frac{1}{\lambda}
	\left(x-\proxi_{\lambda f}(x)\right)$,
	which, together with the firmly non-expansive property (see \cite[Proposition 12.27]{Bauschke2019})
	\[
	(1+\lambda\mu)\nm{\proxi_{\lambda f}(x)-\proxi_{\lambda f}(y)}^2\leqslant \dual{\proxi_{\lambda f}(x)-\proxi_{\lambda f}(y),x-y},
	\]
	implies that $(1+\lambda\mu)f_\lambda$ is $\mu$-strongly convex and $\nabla f_\lambda$ is $1/\lambda$-Lipschitz continuous 
	for any fixed $\lambda>0$.
	\subsection{Energy-conserving solution}
	Following \cite{paoli_existence_2000,schatzman_class_1978}, 
	let us introduce the concept of {\it energy-conserving solution}.
	\begin{Def}\label{def:solu}
		We call $x:[0,\infty)\to \mathcal H$ an energy-conserving
		solution to \cref{eq:ADI} 
		with initial condition \cref{eq:x0-x1} if it satisfies the following.
		\begin{itemize}
			\item $x\in W^{1,\infty}_{loc}(0,\infty;\mathcal H),\,x(0) = x_0$ 
			and $x(t)\in \dom f$ for all $t>0$.
			\item $x'\in BV_{loc}([0,\infty);\mathcal H),\,x'(0+) = x_1$.		
			\item For almost all $t>0$, there holds the energy equality:
			\[
			\begin{aligned}
				{}&	f(x(t)) + \frac{\gamma(t)}2\nm{x'(t)}^2
				+
				\int_{0}^{t}\frac{\mu+3\gamma(s)}{2}\nm{x'(s)}^2
				{\mathrm d}s\\
				= {}& f(x_0) + \frac{\gamma_0}2\nm{x_1}^2+	
				\int_{0}^{t}
				\dual{\xi(s),x'(s)}{\mathrm d}s.
			\end{aligned}
			\]
			\item There exists some $\omega\in M(0,\infty;\mathcal H)$ such that $		\gamma x'' + (\mu+\gamma)x'+\omega = \xi$
			holds in the sense of distributions, and for any $T>0$, we have
			\[
			\int_{0}^{T}\big(f(y(t))-f(x(t))\big){\mathrm d}t
			\geqslant\dual{	\omega,y-x}_{C([0,T];\mathcal H)}
			\quad \text{for all } y\in C([0,T];\mathcal H).
			\]
		\end{itemize}
	\end{Def}
	Let $\lambda>0$ be given and set $F_\lambda = (1+\lambda\mu)f_\lambda$. According to \cref{sec:MY}, we know that $F_\lambda\in\mathcal S_{\mu,\mu+1/\lambda}^{1,1}$. Instead of \cref{eq:ADI}, let us start from a family of regularized problem 
	\begin{equation}\label{eq:regu-ode}
		\gamma x_\lambda''+(\mu+\gamma)x_\lambda'+\nabla F_\lambda(x_\lambda)=\xi,
	\end{equation}
	with initial data $x_\lambda(0) = x_0 $ and $x_\lambda'(0) = x_1$. Since for any fixed $\lambda>0$, $\nabla F_\lambda$ is Lipschitz continuous, 
	applying the standard Picard approximation argument (see the proof of \cite[Theorem 7.3]{brezis_functional_2011}) implies that \cref{eq:regu-ode} admits a unique solution $x_\lambda\in C^{2}([0,\infty);\mathcal H)$.
	
	In the following, we shall establish some priori estimates of $x_\lambda$, 
	then take the limit $\lambda\to0$ to obtain an energy-conserving solution to the original differential inclusion \cref{eq:ADI}.
	\begin{lem}\label{lem:est-x-regu}
		Let $\lambda>0$ be given such that $\lambda\mu\leqslant 1$. For any $T>0$, we have the estimate
		\begin{equation}
			\label{eq:basic-est}
			\nm{x'_\lambda}_{C([0,T];\mathcal H)}+	
			\nm{F_\lambda(x_\lambda) }_{C([0,T])}
			+	\nm{x''_\lambda}_{L^1(0,T;\mathcal H)}+	
			\nm{ \nabla F_\lambda(x_\lambda) }_{L^1(0,T;\mathcal H)}
			\leqslant C_1,
		\end{equation}
		and moreover, if $\dom f = \mathcal H$, then 
		\begin{equation}
			\label{eq:basic-est-2}
			\nm{x''_\lambda}_{C([0,T];\mathcal H)}+	
			\nm{\nabla F_\lambda(x_\lambda) }_{C([0,T];\mathcal H)}
			\leqslant C_2,
		\end{equation}
		where both $C_1$ and $C_2$ are bounded and independent of $\lambda$. 	
	\end{lem}
	\begin{proof}
		Following \cite[Lemma 4.1]{paoli_existence_2000}, we can prove
		\[
		\nm{x'_\lambda}_{C([0,T];\mathcal H)}+	
		\nm{F_\lambda(x_\lambda) }_{C([0,T])}
		\leqslant B_1<\infty,
		\]
		and mimicking the proof of \cite[Lemma 4.2]{paoli_existence_2000}, it is possible to establish
		\[
		\nm{x''_\lambda}_{L^1(0,T;\mathcal H)}+	
		\nm{ \nabla F_\lambda(x_\lambda) }_{L^1(0,T;\mathcal H)}
		\leqslant B_2<\infty.
		\]
		Above, both $B_1$ and $B_2$ are independent of $\lambda$. Collecting them proves \cref{eq:basic-est}.
		
		If $\dom f = \mathcal H$, then $f$ is continuous on $\mathcal H$ since it is convex. From \cref{eq:basic-est}, we conclude that $\nm{x_\lambda}_{C([0,T];\mathcal H)}\leqslant C_3$, where $C_3$ is independent of $\lambda$.
		Therefore, by \cite[Proposition 9.5.2]{attouch_variational_2014}, $\partial f(x_\lambda(t))\neq \emptyset$ is a closed convex bounded set for all $t\in[0,T]$, and there exists some $C_4>0$, which is independent of $\lambda$ as well, such that
		\[
		\sup_{0\leqslant t\leqslant T}\left\{
		\nm{p}:\,p\in \partial f(x_\lambda(t))\right\}\leqslant  C_4.
		\]
		Hence, by \cite[Proposition 17.2.2 (iii)]{attouch_variational_2014}, it follows that
		\[
		\max_{0\leqslant t\leqslant T} \nm{\nabla F_\lambda(x_\lambda(t))}
		=(1+\lambda\mu)	\max_{0\leqslant t\leqslant T} \nm{\nabla f_\lambda(x_\lambda(t))}\leqslant  2C_4,
		\]
		which together with \cref{eq:regu-ode,eq:basic-est} implies 
		\cref{eq:basic-est-2} and finishes the proof.
	\end{proof}
	\begin{thm}\label{thm:regu-x}
		The accelerated differential inclusion \cref{eq:ADI} admits an energy-conserving solution $x:[0,\infty)\to \mathcal H$ in the sense of \cref{def:solu}, and if additionally $\dom f = \mathcal H$, then 
		$x\in W^{2,\infty}_{loc}(0,\infty;\mathcal H)
		\cap C^1([0,\infty);\mathcal H)$ and \cref{eq:ADI} holds for almost all $t>0$.
	\end{thm}
	\begin{proof}
		Given any $T>0$, according to the estimate \cref{eq:basic-est}, it is clear that $\{x_\lambda\}$ is equicontinuous and 
		bounded in $ C([0,T];\mathcal H)$. Invoking the well-known Ascoli--Arzel\`{a} theorem (see \cite[Theorem 6.4 on page 267]{dugundji_topology_1966}), there 
		exists some $x\in C([0,T];\mathcal H)$ and a subsequence which is also denoted by $\{x_\lambda\}$, such that $	x_\lambda\overset{\lambda\to0}\longrightarrow x$ in $C([0,T];\mathcal H)$. As \cref{lem:est-x-regu} also implies that $\{x'_\lambda\}$ is bounded in $L^\infty (0,T;\mathcal H)$, by \cite[Corollary 3.30]{brezis_functional_2011}, there exists $z\in L^\infty(0,T;\mathcal H)$ and a subsequence, such that 
		\begin{equation}\label{eq:dx-regu-conv-Linf}
			\begin{aligned}
				&x'_\lambda\overset{\lambda\to0}\longrightarrow z
				&\quad{\rm weak^*~in~} L^\infty(0,T;\mathcal H).
			\end{aligned}
		\end{equation}
		We conclude immediately that $z = x'$ in the sense of distributions. 
		This means $x\in W^{1,\infty}(0,T;\mathcal H)$. 
		Following \cite[Lemmas 4.3, 4.4 and 4.5]{paoli_existence_2000}, one can verify that $x$ satisfies with all the rest conditions in \cref{def:solu}, 
		and moreover, for almost all $t\in(0,T)$,
		\begin{equation}\label{eq:limit}
			F_\lambda(x_{\lambda}(t))
			\overset{\lambda\to0}{\longrightarrow}
			f(x(t))
			\quad\text{and}\quad
			x'_{\lambda}(t)
			\overset{\lambda\to0}{\longrightarrow}
			x'(t).
		\end{equation}
		
		If $\dom f = \mathcal H$, then by \cref{lem:est-x-regu}, $\{x'_\lambda\}$ 
		is bounded and equicontinuous in $C([0,T];\mathcal H)$. Invoking again the Ascoli--Arzel\`{a} theorem, there exits a subsequence (still denoted by $\{x'_\lambda\}$), such that $x_\lambda'$ converges to $x'$ in $ C([0,T];\mathcal H)$. Moreover, we have $	x''_{\lambda}
		\overset{\lambda\to0}{\longrightarrow}v$ weak$^*$ in $L^\infty(0,T;\mathcal H)$,
		which indicates that $v$ is the generalized derivative of $x'$. Therefore, $x\in W^{2,\infty}(0,T;\mathcal H)\cap C^1([0,T];\mathcal H)$ for 
		any $T>0$. Analogously to \cref{eq:limit}, we have $F_\lambda(x_{\lambda})
		\overset{\lambda\to0}{\longrightarrow}f(x)$ in $C([0,T])$. Now, invoking the proof technique proposed in \cite[Theorem 17.2.2]{attouch_variational_2014}, one can establish
		\[
		f(x) + f^*(\xi-\gamma x''-(\mu+\gamma)x')=\dual{\xi-\gamma x''-(\mu+\gamma)x',x},
		\]
		for almost all $t>0$, where $f^*$ denotes the conjugate function of $f$. 
		By \cite[Proposition 9.5.1]{attouch_variational_2014}, 
		the above identity is equivalent to \cref{eq:ADI}. This finishes the proof.
	\end{proof}
	\subsection{Minimizing property}
	\label{sec:mini}
	In this section, we shall prove the minimizing property of the energy-conserving solution to \cref{eq:ADI} and establish the decay rate by using Lyapunov functions.
	
	Given any $\lambda>0$, let $x_\lambda$ be the unique solution to \cref{eq:regu-ode} and set $v_\lambda = x_{\lambda}+x_\lambda'$, then we define $\mathcal E_\lambda(t):=	{}\mathcal L_\lambda(t)-\Xi_\lambda(t)$, where 
	\begin{numcases}{}
		\label{eq:Xit}
		\Xi_\lambda(t):={}\int_{0}^{t}e^{s-t}
		\dual{\xi(s),v_\lambda(s)-x^*}{\mathrm d}s,\\
		\label{eq:Lt}	
		\mathcal L_\lambda(t):=	{}F_\lambda(x_\lambda(t))-F_\lambda(x^*)
		+\frac{\gamma(t)}{2}
		\nm{v_\lambda(t)-x^*}^2.
	\end{numcases}
	We first establish the minimizing property of $x_\lambda$, then 
	take the limit $\lambda\to0$ to obtain the desired results.
	\begin{lem}\label{lem:conv-Et-regu}
		Given any $\lambda>0$, it holds that 
		\begin{equation}\label{eq:conv-Et-regu}
			\mathcal E'_\lambda(t)
			\leqslant-\mathcal E_\lambda(t)
			-\frac{\mu}{2}\nm{x_\lambda'(t)}^2
			\quad\text{ for all }0\leqslant t<\infty,
		\end{equation}
		which implies
		\begin{equation}\label{eq:conv-Lt-regu}
			\mathcal L_\lambda(t)
			+\frac{\mu}{2}
			\int_{0}^{t}e^{s-t}\nm{x_\lambda'(s)}^2{\mathrm d}s
			\leqslant 
			e^{-t}
			\left(
			2	\mathcal L_\lambda(0)+R^2(t)
			\right),
		\end{equation}
		where $	R(t):=\int_{0}^{t}
		\nm{\xi(s)}
		e^{s/2}/\sqrt{\gamma(s)}
		\,{\mathrm d}s$.
	\end{lem}
	\begin{proof}
		According to \cite[Lemma 3.2]{luo_chen_differential_2019}, it not hard to establish
		\[
		\frac{\dd }{\dd t}\mathcal L_\lambda  \leqslant 
		-\mathcal L_\lambda 	-\frac{\mu}{2}\nm{x_\lambda'}^2
		+\dual{\xi,v_\lambda-x^*},
		\]
		and \cref{eq:conv-Et-regu} follows immediately from this and the trivial identity $	\Xi_\lambda'= 
		\dual{\xi,v_\lambda-x^*}-
		\Xi_\lambda$.
		
		It remains to prove \cref{eq:conv-Lt-regu}.
		By \cref{eq:conv-Et-regu} we have
		\begin{equation}\label{eq:Lt-et}
			e^t		\mathcal E_\lambda(t)+\frac{\mu}{2}
			\int_{0}^{t}e^{s}\nm{x_\lambda'(s)}^2{\mathrm d}s
			\leqslant 
			\mathcal E_\lambda(0)
			=\mathcal L_\lambda(0),
		\end{equation}
		which yields that
		\[
		\begin{aligned}
			{}&\frac{\gamma}{2}\nm{v_\lambda-x^*}^2
			\leqslant \mathcal L_\lambda = 
			\mathcal E_\lambda+\Xi_\lambda
			\leqslant {}\mathcal L_\lambda(0)e^{-t}
			+\int_{0}^{t}e^{s-t}
			\dual{\xi(s),v_\lambda(s)-x^*}
			{\mathrm d}s.
		\end{aligned}
		\]
		Thanks to \cref{lem:Gronwall}, we obtain 
		\[
		\sqrt{\gamma(t)}\nm{v_\lambda(t)-x^*}
		\leqslant
		e^{-t/2}
		\left(	\sqrt{ 2\mathcal L_\lambda(0)}+
		R(t)\right),
		\]
		Based on this, $\Xi_\lambda$ can be estimated as follows
		\[
		\begin{aligned}
			\Xi_\lambda(t) = {}&
			e^{-t}	\int_{0}^{t}e^{s}
			\dual{\xi(s),v_\lambda(s)-x^*}
			{\mathrm d}s\\
			\leqslant {}&
			e^{-t}
			\int_{0}^{t}\frac{e^{s/2}}{\sqrt{\gamma(s)}}
			\nm{\xi(s)}\cdot 
			e^{s/2}\sqrt{\gamma(s)}
			\nm{v_\lambda(s)-x^*}
			{\mathrm d}s\\
			\leqslant {}&
			e^{-t}
			\int_{0}^{t}\frac{e^{s/2}}{\sqrt{\gamma(s)}}
			\nm{\xi(s)}
			\left(
			\sqrt{ 2\mathcal L_\lambda(0)}+R(s)
			\right){\mathrm d}s\\
			={}&
			e^{-t}\left(
			\sqrt{ 2\mathcal L_\lambda(0)}R(t)+\frac{1}{2}R^2(t)\right),
		\end{aligned}
		\]
		which together with \cref{eq:Lt-et} gives
		\[
		\begin{split}	
			\mathcal L_\lambda(t)
			+\frac{\mu}{2}
			\int_{0}^{t}e^{s-t}\nm{x_\lambda'(s)}^2{\mathrm d}s
			\leqslant {}&
			e^{-t}
			\left(
			\mathcal L_\lambda(0)+
			\sqrt{ 2\mathcal L_\lambda(0)}R(t)+
			\frac{1}{2}R^2(t)
			\right)\\
			\leqslant {}&
			e^{-t}
			\left(
			2	\mathcal L_\lambda(0)+R^2(t)
			\right).
		\end{split}
		\]
		This proves \cref{eq:conv-Lt-regu} and completes the proof of this lemma.
	\end{proof}
	
	\begin{thm}\label{thm:conv-f-mu-0}
		If $f\in\mathcal S_0^0$ and $\xi$ satisfies
		\begin{equation}\label{eq:decay-cond-xi-mu-0}
			e^{-t/2}\int_{0}^{t}e^{s}\nm{\xi(s)}{\rm d}s\overset{t\to\infty}{\longrightarrow} 0,
		\end{equation}
		then the accelerated differential inclusion \cref{eq:ADI}
		admits an energy-conserving solution $x:[0,\infty)\to \mathcal H$ satisfying 
		\begin{equation}\label{eq:conv-f-mu-0}
			f(x(t))-f(x^*)
			\leqslant 2	\mathcal L_0e^{-t}
			+\frac{e^{-t}}{\gamma_0}\left(\int_{0}^{t}e^{s}\nm{\xi(s)}{\rm d}s\right)^2,
		\end{equation}		
		for almost all $t>0$, where $	\mathcal L_0 := f(x_0)-f(x^*)
		+\frac{\gamma_0}{2}\nm{x_0+x_1-x^*}^2$.
	\end{thm}
	\begin{proof}
		By \cref{sec:MY}, we have
		\[
		F_\lambda(x_0)-	
		F_\lambda(x^*) = (1+\lambda\mu)(	f_\lambda(x_0)-	
		f_\lambda(x^*) )
		\leqslant (1+\lambda\mu)
		(f(x_0)-f(x^*)),
		\]
		and thus $\mathcal L_\lambda(0)\leqslant (1+\lambda\mu)
		\mathcal L_0$, which together with \cref{eq:conv-Lt-regu} implies 
		\begin{equation}\label{eq:bd-L}
			\mathcal L_\lambda(t)
			\leqslant 
			e^{-t}
			\left(
			2(1+\lambda\mu)	\mathcal L_0+\left(\int_{0}^{t}
			\frac{	e^{s/2}}{\sqrt{\gamma(s)}}
			\nm{\xi(s)}
			\,{\mathrm d}s\right)^2
			\right).
		\end{equation}
		Since $\mu = 0$, by \cref{eq:gamma}, we have $\gamma(t) = \gamma_0 e^{-t}$.
		By the fact \cref{eq:limit}, taking the limit $\lambda\to0$ yields  that
		\[
		f(x(t))-f(x^*)
		+\frac{\gamma(t)}{2}
		\nm{x(t)+x'(t)-x^*}^2
		\leqslant 
		2e^{-t}\mathcal L_0
		+\frac{e^{-t}}{\gamma_0}
		\left(\int_{0}^{t}
		e^{s}
		\nm{\xi(s)}{\mathrm d}s\right)^2,
		\]
		for almost all $t>0$. 
		This implies \cref{eq:conv-f-mu-0} 
		and completes the proof of \cref{thm:conv-f-mu-0}.
	\end{proof}
	\begin{rem}\label{rem:rate-mu-0}
		If $\xi\in L^1_{e^t}(0,\infty;\mathcal H)$, then \cref{eq:decay-cond-xi-mu-0} is satisfied and from \cref{eq:conv-f-mu-0} we obtain the exponential decay
		\[
		f(x(t))-f(x^*)
		\leqslant e^{-t}\left(2	\mathcal L_0
		+\frac{1}{\gamma_0}\nm{\xi}^2_{L^1_{e^t}(0,\infty;\mathcal H)}\right),
		\quad \text{ for almost all } t>0.
		\]
		On the other hand, if $\xi(t)\leqslant e^{-t/2}(t+1)^{-p}$ with $p>0$, then by \cref{lem:est-It},
		\[
		\int_{0}^{t}e^{s}\nm{\xi(s)}{\rm d}s \leqslant 
		\int_{0}^{t}\frac{e^{s/2}}{(t+1)^p}{\rm d}s
		\leqslant \frac{C_pe^{t/2}}{(t+1)^p}.
		\]
		Hence \cref{eq:decay-cond-xi-mu-0} is fulfilled and 
		taking the above estimate into \cref{eq:conv-f-mu-0} yields 
		\[
		f(x(t))-f(x^*)
		\leqslant 
		2\mathcal L_0e^{-t}
		+\frac{ C_p/\gamma_0}{(t+1)^{2p}}.
		\]
	\end{rem}
	\begin{thm}\label{thm:conv-Lt-mu>0}
		If $f\in\mathcal S_\mu^0$ with $\mu>0$ and $\xi$ satisfies
		\begin{equation}\label{eq:decay-cond-xi-mu>0}
			e^{-t/2}\int_{0}^{t}e^{s/2}\nm{\xi(s)}{\rm d}s\overset{t\to\infty}{\longrightarrow} 0,
		\end{equation}
		then the accelerated differential inclusion \cref{eq:ADI}
		admits an energy-conserving solution $x:[0,\infty)\to \mathcal H$ 
		satisfying
		\begin{equation}\label{eq:conv-Lt-mu>0}
			f(x(t))-f(x^*)
			+\frac{\gamma(t)}{2}
			\nm{x(t)+x'(t)-x^*}^2
			\leqslant 2	\mathcal L_0e^{-t}+	
			\frac{e^{-t}}{\gamma_{\min}}\left(\int_{0}^{t}e^{s/2}\nm{\xi(s)}{\rm d}s\right)^2,
		\end{equation}		
		for almost all $t>0$, where $	\mathcal L_0 = f(x_0)-f(x^*)
		+\frac{\gamma_0}{2}\nm{x_0+x_1-x^*}^2$ and $\gamma_{\min} = \min\{\mu,\gamma_0\}>0$.
	\end{thm}
	\begin{proof}
		As $\mu>0$, by the fact $\gamma(t)\geqslant \gamma_{\min}$ and \cref{lem:est-It}, we have
		\[
		\begin{aligned}
			{}&
			\int_{0}^{t}
			\frac{e^{s/2}}{\sqrt{\gamma(s)}}
			\nm{\xi(s)}{\mathrm d}s\leqslant 
			\frac{1}{\sqrt{\gamma_{\min}}}
			\int_{0}^{t}\frac{e^{s/2}}{(s+1)^p}\,{\mathrm d}s
			\leqslant{} 
			\frac{ C_p}{\sqrt{\gamma_{\min}}}
			\cdot
			\frac{ e^{t/2}}{(t+1)^p},
		\end{aligned}
		\]
		and thus, by \cref{eq:limit,eq:bd-L}, taking the limit $\lambda\to0$ gives
		\[
		f(x(t))-f(x^*)
		+\frac{\gamma(t)}{2}
		\nm{x(t)+x'(t)-x^*}^2
		\leqslant 
		2\mathcal L_0e^{-t}
		+\frac{ C_p/\gamma_{\min}}{(t+1)^{2p}},
		\]
		for almost all $t>0$, which proves \cref{eq:conv-Lt-mu>0} and 
		concludes the proof of \cref{thm:conv-Lt-mu>0}.
	\end{proof}
	\begin{rem}\label{rem:rate-mu>0}
		If $\xi\in L^1_{e^{t/2}}(0,\infty;\mathcal H)$, then \cref{eq:decay-cond-xi-mu>0} holds true and by \cref{eq:conv-Lt-mu>0} we have the exponential decay
		\[
		f(x(t))-f(x^*)
		+\frac{\gamma(t)}{2}
		\nm{x(t)+x'(t)-x^*}^2
		\leqslant e^{-t}\left(2	\mathcal L_0
		+\frac{1}{\gamma_{\min}}\nm{\xi}^2_{L^1_{e^{t/2}}(0,\infty;\mathcal H)}\right),
		\]
		for almost all $t>0$. Besides, if $\xi(t)\leqslant (t+1)^{-p}$ with $p>0$, 
		then invoking \cref{lem:est-It} gives
		\[
		\int_{0}^{t}e^{s/2}\nm{\xi(s)}{\rm d}s \leqslant 
		\int_{0}^{t}\frac{e^{s/2}}{(t+1)^p}{\rm d}s
		\leqslant \frac{C_pe^{t/2}}{(t+1)^p}.
		\]
		This promises \cref{eq:decay-cond-xi-mu>0} and by \cref{eq:conv-Lt-mu>0} we find that
		\begin{equation}\label{eq:rate-mu>0}
			f(x(t))-f(x^*)
			+\frac{\gamma(t)}{2}
			\nm{x(t)+x'(t)-x^*}^2
			\leqslant 
			2\mathcal L_0e^{-t}
			+\frac{ C_p/\gamma_{\min}}{(t+1)^{2p}}.
		\end{equation}
	\end{rem}
	\section{An Inexact Accelerated PPA}
	\label{sec:AIPPA}
	
	In this section, we are mainly interested in the case 
	$f\in\mathcal S_\mu^0$ with $\mu\geqslant 0$. Rewrite 
	\cref{eq:ADI} as a first-order differential inclusion system
	\begin{subnumcases}{}
		\label{eq:ADI-x}
		{~\,}x' = v-x,\\
		\gamma v' \in \mu(x-v)-\partial f(x)+\xi,
		\label{eq:ADI-v}
	\end{subnumcases}
	and consider the implicit Euler scheme
	\begin{subnumcases}{}
		\label{eq:im-ADI-xk1}
		\frac{x_{k+1}-x_k}{\alpha_k} = v_{k+1}-x_{k+1},\\
		\gamma_k 	\frac{v_{k+1}-v_k}{\alpha_k} \in 
		\mu(x_{k+1}-v_{k+1})
		-\partial f(x_{k+1})+\xi_{k},
		\label{eq:im-ADI-vk1}
	\end{subnumcases}
	where the equation \cref{eq:gamma} of the scaling factor 
	$\gamma$ is also discretized implicitly
	\begin{equation}\label{eq:gammak}
		\frac{\gamma_{k+1}-\gamma_k}{\alpha_k} 
		= \mu-\gamma_{k+1},\quad\gamma_0>0.
	\end{equation}
	After some
	manipulations, we rewrite \cref{eq:im-ADI-xk1} as follows
	\begin{subnumcases}{}
		x_{k+1} = \proxi_{\lambda_k f}(w_k+\lambda_k\xi_{k}),\,\,\lambda_k = {}\alpha_k^2/\eta_k,
		\label{eq:im-ADI-equi-xk1}\\
		\label{eq:im-ADI-equi-vk1}
		v_{k+1}=x_{k+1}+	\frac{x_{k+1}-x_k}{\alpha_k},
	\end{subnumcases}
	where 
	\begin{equation}\label{eq:lambdak-wk}
		w_k ={} \frac1{\eta_k}(\gamma_k\alpha_kv_k+(\gamma_k+\mu\alpha_k)x_k),\quad \eta_k=\gamma_k\alpha_k+\gamma_k+\mu\alpha_k.
	\end{equation}
	
	Note that \cref{eq:im-ADI-equi-xk1} gives a PPA with extrapolation. The additional term $\lambda_k\xi_{k}$ makes the proximal step \eqref{eq:im-ADI-equi-xk1} inexact and it is more general to consider 
	\begin{subnumcases}{}
		x_{k+1} \approx \proxi_{\lambda_k f}(w_k),
		\label{eq:AIPPA-xk1}\\
		\label{eq:AIPPA-vk1}
		v_{k+1}=x_{k+1}+	\frac{x_{k+1}-x_k}{\alpha_k}.
	\end{subnumcases}
	In the state of the art, we have many choices for \eqref{eq:AIPPA-xk1}; see the  inexact accelerated proximal point algorithms proposed in \cite{guler_new_1992,salzo_inexact_2012}. 
	In the forthcoming section, we shall list some typical approximations 
	to the proximal mapping $\proxi_{\lambda f}$.
	\subsection{Inexact proximal mapping}
	\label{sec:inexact-proxi}
	In \cite{salzo_inexact_2012}, Salzo and Villa summarized three 
	types of approximations from \cite{alber_proximal_1997,auslender_numerical_1987,rockafellar_monotone_1976}.
	The first two of them involve the concept of $\varepsilon$-$subdi\!f\!\!f\!erential$:
	\begin{equation*}
		\partial f(x,\varepsilon):=
		\left\{p\in \mathcal H:f(y)\geqslant f(x)+\dual{p,x-y}-\varepsilon\quad\forall\,y\in \mathcal H  \right\}.
	\end{equation*}
	By definition, it is clear that $0\in\partial f(x,\varepsilon)$ iff $f(x)\leqslant \varepsilon+\inf_{y\in \mathcal H}f(y)$.
	\begin{Def}[\cite{alber_proximal_1997}]
		\label{def:type1}
		Given $\varepsilon,\,\lambda>0$ and $x\in \mathcal H$, we call $w\in \mathcal H$ a 
		$type$-1 approximation of $\proxi_{\lambda f}(x)$ 
		with $\varepsilon$-precision and write
		$w=_{1,\varepsilon}\!\proxi_{\lambda f}(x)$, if and 
		only if $	0\in \partial \phi_\lambda(w,\varepsilon^2/(2\lambda))$,
		where $\phi_\lambda(\cdot) = f(\cdot)+\frac{1}{2\lambda}\nm{\cdot-x}^2$. Equivalently, $w=_{1,\varepsilon}\!\proxi_{\lambda f}(x)$ if and only if 
		\begin{equation*}
			f(w)+\frac{1}{2\lambda}\nm{w-x}^2
			\leqslant 
			\frac{\varepsilon^2}{2\lambda}+f_\lambda (x).
		\end{equation*}
	\end{Def}
	\begin{Def}[\cite{auslender_numerical_1987}]
		\label{def:type2}
		Given $\varepsilon,\,\lambda>0$ and $x\in \mathcal H$, we call $w\in \mathcal H$ a $type$-2 approximation of $\proxi_{\lambda f}(x)$ 
		with $\varepsilon$-precision and write
		$w=_{2,\varepsilon}\!\proxi_{\lambda f}(x)$, if and only if 
		\begin{equation*}
			\frac{x-w}{\lambda} \in \partial f(w,\varepsilon^2/(2\lambda)).
		\end{equation*}
	\end{Def}
	\begin{Def}[\cite{rockafellar_monotone_1976}]
		\label{def:type3}
		Given $\varepsilon,\,\lambda>0$ and $x\in \mathcal H$, we call $w\in \mathcal H$ a 
		$type$-3 approximation of $\proxi_{\lambda f}(x)$ 
		with $\varepsilon$-precision and write
		$w=_{3,\varepsilon}\!\proxi_{\lambda f}(x)$, iff $	{\rm dist}\left(0,\partial \phi_\lambda(w)\right)\leqslant \frac{\varepsilon}{\lambda}$,
		where  $		{\rm dist}\left(0,\partial \phi_\lambda(w)\right) := 
		\inf\left\{
		\nm{p}:p\in \partial \phi_\lambda(w)
		\right\}$.
	\end{Def}
	For practical approximations, we refer to \cite{he_accelerated_2012}.
	Some variant of $type$-2 approximation can be found in \cite{monteiro_convergence_2010}. The following result, 
	coming from \cite[Proposition 1]{salzo_inexact_2012},
	compares those three kinds of approximations. For more implications under 
	further assumption, e.g., boundness of $\dom f$, see \cite[Proposition 2.4]{villa_accelerated_2013}.
	\begin{prop}[\cite{salzo_inexact_2012}]
		\label{prop:equiv}
		The following implications are true:
		\begin{enumerate}
			\item $w=_{2,\varepsilon}\!\proxi_{\lambda f}(x)\,\Longrightarrow
			\,w=_{1,\varepsilon}\!\proxi_{\lambda f}(x)$;
			\item $w=_{3,\varepsilon}\!\proxi_{\lambda f}(x)\,\Longrightarrow
			\,w=_{1,\varepsilon}\!\proxi_{\lambda f}(x)$;
			\item $w=_{3,\varepsilon}\!\proxi_{\lambda f}(x)\,\Longleftrightarrow
			\,\!w=\proxi_{\lambda f}(x+ e)$ with some $e\in \mathcal H$ such that 
			$\nm{e}\leqslant \varepsilon$.
		\end{enumerate}
	\end{prop}
	Therefore, both $type$-2 and $type$-3 approximations can be reduced 
	to $type$-1 approximation and the scheme \cref{eq:im-ADI-equi-xk1}, as well as the implicit Euler discretization \cref{eq:im-ADI-xk1}, considers $type$-3 approximation. However, with same magnitude error $\varepsilon$, 
	the corresponding PPA using $type$-2 and $type$-3 approximations has faster convergence rate. In other words, the decay assumption on $\varepsilon$ of $type$-2 and $type$-3 approximations is weaker than that of $type$-1 approximation.
	
	Indeed, for convex $f$, G\"{u}ler \cite{guler_new_1992} proposed an inexact accelerated PPA involving $type$-3 approximation: $x_{k+1}=_{3,\varepsilon_k}\!\proxi_{\lambda_k f}(w_k)$; see \cite[Section 2]{guler_new_1992}. The rate $O(1/k^2+k^{1-2p})$ has been established (cf. \cite[Theorem 3.3]{guler_new_1992}) with 
	\begin{equation}\label{eq:assum}
		\lambda_k=\lambda>0 \quad\text{and}\quad\varepsilon_{k} = C(k+1)^{-p}.
	\end{equation}
	Salzo and Villa \cite{salzo_inexact_2012} presented an inexact accelerated
	PPA using $type$-2 approximation: $x_{k+1}=_{2,\varepsilon_k}\!\!\proxi_{\lambda_k f}(w_k)$ and they 
	obtained the rate $O(1/k^2+k^{1-2p})$ as well, under assumption \cref{eq:assum}. Besides, they considered an inexact PPA that adopts $type$-1 approximation: $x_{k+1}=_{1,\varepsilon_k}\!\proxi_{\lambda_k f}(w_k)$. However, with the same condition \cref{eq:assum}, they only proved the rate $O(1/k+k^{3-2p})$  (see \cite[Theorem 4]{salzo_inexact_2012}).
	
	In the next section, we shall apply $type$-1 approximation to our inexact accelerated PPA \cref{eq:AIPPA-xk1}. Under assumption \cref{eq:assum}, for convex case ($\mu=0$), we derive the convergence rate $O(1/k^2+k^{2-2p})$ (cf. \cref{thm:conv-AIPPA-mu-0})
	which improves the result in \cite[Theorem 4]{salzo_inexact_2012}, and for strongly 
	convex case ($\mu>0$), we take constant 
	step size  $\alpha_k=\alpha>0$ and obtain new convergence rate estimate (cf. \cref{thm:conv-AIPPA-mu>0}).
	
	To the end, we list a key lemma, which is important for our 
	convergence rate analysis; see \cite[Lemma 2.7]{Lin2020} for a detailed proof.
	\begin{lem}[\cite{Lin2020}]
		\label{lem:key-type1}
		Assume $f\in\mathcal S_{\mu}^0$ with $\mu\geqslant 0$ 
		and let $\varepsilon,\,\lambda>0$ and $x\in \mathcal H$ be given.
		Then for $w=_{1,\varepsilon}\!\proxi_{\lambda f}(x)$, it holds that
		\begin{equation*}
			\begin{split}
				\frac{\varepsilon^2}{2\lambda}
				+f(y)
				\geqslant {}&f(w)+
				\frac{1}{\lambda}
				\dual{\sigma+w-x,w-y}
				+\frac{\mu}{2}\nm{w-y}^2+
				\frac{1+\lambda\mu}{2\lambda}\nm{w-\proxi_{\lambda f}(x)}^2,
			\end{split}
		\end{equation*}
		for all $y\in \mathcal H$, where $\sigma=(1+\lambda\mu)(\proxi_{\lambda f}(x)-w)$. 
		Particularly, taking $y=w$ gives
		\[
		\nm{w-\proxi_{\lambda f}(x)}\leqslant   \frac{\varepsilon}{\sqrt{1+\lambda\mu}}.
		\]
	\end{lem}
	\subsection{Main results and proofs}
	Given a nonnegative sequence $\{\varepsilon_k\}_{k = 0}^{\infty}$, 
	let us consider the scheme \cref{eq:AIPPA-xk1} with $type$-1 approximation:
	\begin{subnumcases}{}
		x_{k+1} =_{1,\varepsilon_k}\!\proxi_{\lambda_k f}(w_k)
		,\,\,\lambda_k = {}\alpha_k^2/\eta_k,
		\label{eq:AIPPA-xk1-type1}\\
		\label{eq:AIPPA-vk1-type1}
		v_{k+1}=x_{k+1}+	\frac{x_{k+1}-x_k}{\alpha_k},
	\end{subnumcases}
	with $\eta_k$ and $w_k$ being defined by \cref{eq:lambdak-wk}. 
	Recall that the equation \cref{eq:gamma} of the scaling factor 
	$\gamma$ is discretized implicitly by \cref{eq:gammak}.
	It is worth noticing that \cref{eq:AIPPA-xk1-type1} implies the identity
	\begin{equation}\label{eq:IAPA-difff-k1}
		\frac{v_{k+1}-v_{k}}{\alpha_k}={}
		\frac{\mu}{\gamma_k}  (x_{k+1}-v_{k+1})
		-\frac{1}{\gamma_k}
		\frac{w_k-x_{k+1}}{\lambda_k}.
	\end{equation}
	To verify this, let $(\widehat{x}_{k+1},\widehat{v}_{k+1})$ satisfy
	\[
	\left\{
	\begin{aligned}
		\frac{\widehat{x}_{k+1}-x_k}{\alpha_k} = {}&\widehat{v}_{k+1}-\widehat{x}_{k+1},\\
		\gamma_k 	\frac{\widehat{v}_{k+1}-v_k}{\alpha_k} \in {}&
		\mu(\widehat{x}_{k+1}-\widehat{v}_{k+1})
		-\partial f(\widehat{x}_{k+1}),
	\end{aligned}
	\right.
	\]
	then it is evident that $\widehat{x}_{k+1} ={}\proxi_{\lambda_k f}(w_k)$, where $\lambda_k$ and $w_k$ are the same as that in \cref{eq:AIPPA-vk1-type1}. Based on this and \cref{eq:sub-inf}, a direct manipulation verifies \cref{eq:IAPA-difff-k1}.
	
	Assume that 
	\begin{equation}\label{eq:decay-epsilonk}
		\varepsilon_k\leqslant \frac{1}{(k+1)^p},\quad p>0.
	\end{equation}
	Our main estimates are listed below.
	\begin{thm}
		\label{thm:conv-AIPPA-mu-0}
		Suppose $f\in\mathcal S_0^{0}$, i.e., $\mu=0$. If (i) $\gamma_0=4$, (ii) $\varepsilon_{k}$ obeys \cref{eq:decay-epsilonk} with $p>1$ and (iii) $\alpha_k^2=\gamma_k(1+\alpha_k)$, then for the scheme \cref{eq:AIPPA-xk1-type1}, we have $\lambda_k = 1$ and 
		\begin{equation}
			\label{eq:conv-AIPPA-mu-0}
			f(x_k)-f(x^*)\leqslant 
			\frac{2\mathcal L_0}{(k+1)^2} 
			+ C_p\left\{
			\begin{aligned}
				&\frac{1+\ln^2(k+1)}{(k+1)^{2}},&&p=2,\\
				&(k+1)^{-2}+(k+1)^{2-2p},&&p>1 \text{ and } p\neq 2,
			\end{aligned}
			\right.
		\end{equation}
		where $\mathcal L_0: = 
		f(x_{0})-f(x^*) + 
		\frac{\gamma_{0}}{2}
		\nm{v_{0}-x^*}^2$.
	\end{thm}
	\begin{rem}
		In the setting of \cref{thm:conv-AIPPA-mu-0} (except \cref{eq:decay-epsilonk}), if  $\{k\varepsilon_k\}_{k=1}^{\infty}$ is summable, then we have
		the rate $O(1/k^2)$ without log factor $\ln k$, which is promised by  \cref{eq:key-mu-0}.
	\end{rem}
	\begin{thm}
		\label{thm:conv-AIPPA-mu>0}
		Suppose $f\in\mathcal S_\mu^{0}$ with $\mu>0$. If (i) $\gamma_0 = \mu$, (ii) $\varepsilon_{k}$ 
		satisfies \cref{eq:decay-epsilonk} and we choose $\alpha_k=\alpha>0$, 
		then for the scheme \cref{eq:AIPPA-xk1-type1}, we have $\lambda_k= {}\alpha^2/(\mu+2\mu\alpha)$ and 
		\begin{equation}\label{eq:conv-AIPPA-mu>0}
			f(x_k)-f(x^*)+\frac{\mu}{2}\nm{v_k-x^*}^2
			\leqslant 
			\frac{2\mathcal L_0}{(1+\alpha)^k} 
			+\frac{C_{\alpha,p,\mu}}{(k+1)^{2p}},
		\end{equation}	
		where $\mathcal L_0: = 
		f(x_{0})-f(x^*) + 
		\frac{\gamma_{0}}{2}
		\nm{v_{0}-x^*}^2$.
	\end{thm}
	\begin{rem}
		Note that the estimate \cref{eq:conv-AIPPA-mu>0} matches the continuous decay rate \cref{eq:rate-mu>0}.
		In addition, under the assumption of \cref{thm:conv-AIPPA-mu>0} (except \cref{eq:decay-epsilonk}), if we impose stronger condition that $\{(1+\alpha)^{k/2}\varepsilon_k\}_{k=0}^{\infty}$ is summable, then by \cref{eq:key-mu>0,lem:conv-Lk-AIPPA}, 
		we have linear convergence rate $O((1+\alpha)^{-k})$.
	\end{rem}
	To prove the above two theorems, we need several preparations.
	Let $\{(\gamma_k,x_k,v_k)\}_{k=0}^{\infty}$ be generated 
	by \cref{eq:gammak,eq:AIPPA-xk1-type1}.
	Define a discrete Lyapunov function 
	\begin{equation}\label{eq:Lk}
		\mathcal L_k: = 
		f(x_{k})-f(x^*) + 
		\frac{\gamma_{k}}{2}
		\nm{v_{k}-x^*}^2
		\quad \text{ for all } k\in\mathbb N,
	\end{equation}
	and we also need two error accumulation functions 
	\[
	\left\{
	\begin{split}
		{}&	\Upsilon_0 ={} 0,\quad\Upsilon_k =\frac{1}{2} \sum_{i=0}^{k-1}\frac{\varepsilon_i^2}{\lambda_{i}\beta_{i+1}},\quad k\geqslant 1,\\
		{}&\Omega_0={}0,\quad
		\Omega_k = \sum_{i=0}^{k-1}\frac{\alpha_i}{\lambda_i\beta_{i}}\cdot \sqrt{\frac{\beta_{i+1}}{\gamma_{i+1}}}\nm{\xi_{i+1}},\quad k\geqslant 1,
	\end{split}
	\right.
	\]
	where 
	$\xi_{k+1}=(1+\lambda_k\mu)(\proxi_{\lambda_k f}(w_k)-x_{k+1})$ and
	\begin{equation}
		\label{eq:betak}
		\beta_0 = {}1,\quad 
		\beta_{k} = \prod_{i=0}^{k-1}\frac{1}{1+\alpha_i},\quad k\geqslant 1.
	\end{equation}
	\begin{lem}\label{lem:conv-Lk-AIPPA}
		For the scheme \cref{eq:AIPPA-xk1-type1}, we have
		\begin{equation}\label{eq:conv-Lk-AIPPA}
			\mathcal L_k\leqslant 2\beta_k (\mathcal L_0+\Upsilon_k +\Omega^2_k)
			\quad \text{ for all } k\in\mathbb N.
		\end{equation}
	\end{lem}
	\begin{proof}	
		Collecting \cref{eq:AIPPA-vk1-type1,eq:IAPA-difff-k1}, we see
		\[
		\left\{
		\begin{split}
			\frac{x_{k+1}-x_k}{\alpha_k}={}&	v_{k+1}-x_{k+1},\\
			\frac{v_{k+1}-v_{k}}{\alpha_k}={}&
			\frac{\mu}{\gamma_k}  (x_{k+1}-v_{k+1})
			-\frac{1}{\gamma_k}
			\frac{w_k-x_{k+1}}{\lambda_k}.
		\end{split}
		\right.
		\]
		Thanks to this discretization formulation, we can follow the deduction from 
		\cite[Theorem 3.1]{luo_chen_differential_2019} to obtain the difference
		\[
		\begin{split}
			\mathcal L_{k+1}-\mathcal L_{k} 
			={}&f(x_{k+1})-f(x_k) -
			\frac{\alpha_k\gamma_{k+1}}{2}
			\nm{v_{k+1}-x^*}^2- 
			\frac{\gamma_k}2
			\nm{v_{k+1}-v_k}^2\\
			{}&\quad- 	\frac{1}{\lambda_k}\dual{ w_k-x_{k+1}, x_{k+1} - x_{k}} 
			-	\frac{\alpha_k}{\lambda_k}\dual{ w_k-x_{k+1}, x_{k+1} - x^{*}}\\
			{}&\quad\qquad +
			\frac{	\mu\alpha_k}{2}\left(\left\|x_{k+1}-x^{*}\right\|^{2}-\|x_{k+1}-v_{k+1}\|^{2}\right).
		\end{split}
		\]
		Invoking \cref{lem:key-type1}, we have
		\[
		\begin{split}
			- 	\frac{1}{\lambda_k}\dual{ w_k-x_{k+1}, x_{k+1} - x_{k}} \leqslant {}&\frac{\varepsilon_k^2}{2\lambda_k}+f(x_k)-f(x_{k+1})
			-\frac{1}{\lambda_k}\dual{\xi_{k+1},x_{k+1}-x_k},
		\end{split}
		\]
		and
		\[
		\begin{split}
			{}&-	\frac{1}{\lambda_k}\dual{ w_k-x_{k+1}, x_{k+1} - x^{*}}\\
			\leqslant
			{}&\frac{\varepsilon_k^2}{2\lambda_k}+f(x^*)-f(x_{k+1})
			-\frac{1}{\lambda_k}\dual{\xi_{k+1},x_{k+1}-x^*}-
			\frac{\mu}2\nm{x_{k+1}-x^*}^2.
		\end{split}
		\]
		Consequently, it follows that
		\[
		\begin{split}
			\mathcal L_{k+1}-\mathcal L_{k} 
			\leqslant{}&
			-\alpha_k \mathcal L_{k+1}+\frac{1+\alpha_k}{2\lambda_k}\varepsilon_k^2
			-\frac{1}{\lambda_k}\dual{\xi_{k+1},x_{k+1}-x_k+\alpha_k(x_{k+1}-x^*)}\\
			{}&\quad -\frac{	\mu\alpha_k}{2}\|x_{k+1}-v_{k+1}\|^{2}- 
			\frac{\gamma_k}2\nm{v_{k+1}-v_k}^2.
		\end{split}
		\]
		Dropping surplus nonpositive terms and using \cref{eq:AIPPA-vk1-type1}, we 
		arrive at
		\[
		\begin{split}
			\mathcal L_{k+1}-\mathcal L_{k} 
			\leqslant{}&
			-\alpha_k \mathcal L_{k+1}+\frac{1+\alpha_k}{2\lambda_k}\varepsilon_k^2
			-\frac{\alpha_k}{\lambda_k}\dual{\xi_{k+1},v_{k+1}-x^*}\\
			\leqslant{}&
			-\alpha_k \mathcal L_{k+1}+\frac{1+\alpha_k}{2\lambda_k}\varepsilon_k^2
			+\frac{\alpha_k}{\lambda_k}\nm{\xi_{k+1}}\nm{v_{k+1}-x^*}.
		\end{split}
		\]
		Recursively, it holds that	
		\begin{equation}\label{eq:Lk-1}
			\begin{split}
				\mathcal L_k\leqslant{}& \beta_{k}\mathcal L_0+\beta_{k}
				\sum_{i=0}^{k-1}\frac{1}{\beta_{i}}
				\left(
				\frac{1+\alpha_i}{2\lambda_i}\varepsilon_i^2
				+\frac{\alpha_i}{\lambda_i}\nm{\xi_{i+1}}\nm{v_{i+1}-x^*}
				\right)\\
				={}&\beta_k(\mathcal{L}_0+\Upsilon_k)
				+\beta_{k}\sum_{i=0}^{k-1}\frac{\alpha_i}{\lambda_i\beta_{i}}
				\nm{\xi_{i+1}}\nm{v_{i+1}-x^*}.
			\end{split}
		\end{equation}
		
		Note that above inequality implies 
		\[
		\frac{\gamma_k}{2}\nm{v_k-x^*}^2\leqslant \beta_k\left(\mathcal L_0+\Upsilon_k\right) +\beta_{k}\sum_{i=0}^{k-1}\frac{\alpha_i}{\lambda_i\beta_{i}}
		\nm{\xi_{i+1}}\nm{v_{i+1}-x^*}.
		\]
		Therefore, utilizing \cref{lem:Gronwall-discrete} gives
		\[
		\sqrt{\frac{\gamma_k}{\beta_k}}\nm{v_k-x^*}\leqslant \sqrt{2(\mathcal L_0+\Upsilon_k)}
		+2\sum_{i=0}^{k-1}
		\frac{\alpha_i\nm{\xi_{i+1}}}{\lambda_i\beta_{i}}
		\sqrt{\frac{\beta_{i+1}}{\gamma_{i+1}}}
		=\sqrt{2(\mathcal L_0+\Upsilon_k)}+2\Omega_k.
		\]
		This can be reused for the estimate
		\begin{equation}\label{eq:est-cross}
			\small
			\begin{aligned}
				{}&  
				\sum_{i=0}^{k-1}
				\frac{\alpha_i}{\lambda_i\beta_{i}}
				\nm{\xi_{i+1}}\nm{v_{i+1}-x^*}
				\leqslant {} \sum_{i=0}^{k-1}
				\frac{\alpha_i\nm{\xi_{i+1}}}{\lambda_i\beta_{i}}\cdot
				\sqrt{\frac{\beta_{i+1}}{\gamma_{i+1}}}
				\cdot
				\left(\sqrt{2(\mathcal L_0+\Upsilon_i)}
				+2\Omega_i
				\right)\\
				= {}& \Omega_k^2+  \sum_{i=0}^{k-1}
				\frac{\alpha_i\nm{\xi_{i+1}}}{\lambda_i\beta_{i}}\cdot
				\sqrt{\frac{\beta_{i+1}}{\gamma_{i+1}}}
				\cdot\sqrt{2(\mathcal L_0+\Upsilon_i)}
				\leqslant  2\Omega_k^2+\mathcal L_0+\Upsilon_k.
			\end{aligned}
		\end{equation}
		This together with \cref{eq:Lk-1} gives \cref{eq:conv-Lk-AIPPA} and concludes the proof.
	\end{proof}
	We now arrive at a position for proving \cref{thm:conv-AIPPA-mu-0,thm:conv-AIPPA-mu>0}.
	
	\medskip\noindent{\bf Proof of \cref{thm:conv-AIPPA-mu-0}.} 
	By \cref{lem:key-type1}, we have $\nm{\xi_{k+1}}\leqslant \varepsilon_{k}$. In view of \cref{eq:lambdak-wk} and our choice $\alpha_k^2=\gamma_k(1+\alpha_k)$, 
	it is easy to get $\lambda_k = 1$. Observing \cref{eq:gammak}, we conclude 
	that $\beta_{k}=\gamma_{k}/\gamma_0$, and thus, it follows 
	\[
	\begin{split}
		\Upsilon_k= {}&\frac{\gamma_0}{2} \sum_{i=0}^{k-1}\frac{\varepsilon_i^2}{\gamma_{i+1}},\quad 
		\Omega_k ={} \sqrt{\gamma_0}\sum_{i=0}^{k-1}\frac{\nm{\xi_{i+1}}}{\sqrt{\gamma_{i+1}}}
		\leqslant \sqrt{\gamma_0}\sum_{i=0}^{k-1}
		\frac{\varepsilon_i}{\sqrt{\gamma_{i+1}}}.
	\end{split}
	\]
	As \cref{lem:gammak-ak} implies 
	\[
	\frac{\gamma_0}{(\sqrt{\gamma_0} \, k +1)^2}
	\leqslant 
	\gamma_k\leqslant
	\frac{4\gamma_0}{(\sqrt{\gamma_0} \, k + 2)^2},
	\]
	we obtain from \cref{lem:conv-Lk-AIPPA} and the choice $\gamma_0=4$ that
	\begin{equation}\label{eq:key-mu-0}
		\mathcal L_k\leqslant \frac{2\mathcal L_0}{(k+1)^2} +\frac{9}{(k+1)^2}\sum_{i=0}^{k-1}(i+1)^{2}\varepsilon_i^2
		+\frac{18}{(k+1)^2}\left(\sum_{i=0}^{k-1}(i+1)\varepsilon_i\right)^2.
	\end{equation}
	Since $\varepsilon_{k}$ satisfies \cref{eq:decay-epsilonk}, 
	invoking the following two elementary estimates:
	\begin{equation}\label{eq:1-p}
		\begin{split}
			{}&\sum_{i=0}^{k-1} (i+1)^{1-p}
			\leqslant1+\int_{0}^{k}(s+1)^{1-p}\dd s
			\leqslant {}
			\left\{
			\begin{aligned}
				&1+\ln (k+1),&&\text{ if } p=2,\\
				&C_{p}\left(1+(k+1)^{2-p}\right),&&\text{ else},
			\end{aligned}
			\right.
		\end{split}
	\end{equation}
	and 
	\begin{equation}\label{eq:2-2p}
		\begin{split}
			{}&\sum_{i=0}^{k-1} (i+1)^{2-2p}
			\leqslant1+\int_{0}^{k}(s+1)^{2-2p}\dd s
			\leqslant {}
			\left\{
			\begin{aligned}
				&1+\ln (k+1),&&\text{ if } p=3/2,\\
				&C_{p}\left(1+(k+1)^{3-2p}\right),&&\text{ else},
			\end{aligned}
			\right.
		\end{split}
	\end{equation}
	we finally establish \cref{eq:conv-AIPPA-mu-0} and finish the proof of \cref{thm:conv-AIPPA-mu-0}. \qed
	
	\medskip\noindent{\bf Proof of \cref{thm:conv-AIPPA-mu>0}.} 
	Since $\gamma_0=\mu$, according to \cref{lem:gammak-ak}, $\gamma_k=\mu$ and thus by \cref{eq:lambdak-wk}, $\lambda_k=\lambda$. Moreover, by \cref{lem:key-type1}, $\nm{\xi_{k+1}}\leqslant \varepsilon_{k}\sqrt{1+\lambda\mu}\leqslant \varepsilon_{k}\sqrt{1+\alpha}$. It follows immediately that
	\begin{equation}\label{eq:key-mu>0}
		\Upsilon_k= {}\frac{1}{2\lambda} \sum_{i=0}^{k-1}(1+\alpha)^{i+1}\varepsilon_i^2,\quad 
		\Omega_k \leqslant {}\frac{\sqrt{\alpha/\mu}}{\lambda}\sum_{i=0}^{k-1}(1+\alpha)^{(i+1)/2}\varepsilon_i,
	\end{equation}
	and we use \cref{eq:decay-epsilonk,lem:conv-Lk-AIPPA} to get
	\[
	\mathcal L_k\leqslant \frac{2\mathcal L_0}{(1+\alpha)^k}+
	\frac{C_{\alpha,p,\mu}}{(1+\alpha)^k}\left[\sum_{i=1}^{k}\frac{(1+\alpha)^i}{(i+1)^{2p}}
	+\left(\sum_{i=1}^{k}\frac{(1+\alpha)^{i/2}}{(i+1)^p}\right)^2\right].
	\]
	Thus, applying \cref{lem:estk-It} proves \cref{eq:conv-AIPPA-mu>0} and 
	concludes the proof of \cref{thm:conv-AIPPA-mu>0}. \qed
	\section{An Inexact Accelerated PGM}
	\label{sec:AIPGM}
	We now move to the composite case 
	$f = h+g$ where the smooth part $h\in\mathcal S_{\mu,L}^{1,1}$ with $0\leqslant \mu\leqslant L<\infty$ and the nonsmooth part $g\in\mathcal S_0^0$. 
	In this case, the first-order system \cref{eq:ADI-x} becomes 
	\begin{subnumcases}{}
		\label{eq:h-g-ADI-x}
		{~\,}x' = v-x,\\
		\gamma v' \in \mu(x-v)-\nabla h(x)-\partial g(x)+\xi.
		\label{eq:h-g-ADI-v}
	\end{subnumcases}
	
	In the last section, we investigated an implicit scheme that involves 
	the (approximated) gradient mapping $\proxi_{\lambda_k f}$. 
	However, for the current separable case $f = h+g$, it would be 
	more efficient to use semi-implicit scheme, which means we use explicit scheme for smooth part $h$ and implicit scheme for nonsmooth part $g$. In other words, consider the splitting method
	\begin{subnumcases}{}
		\label{eq:semi-ADI-xk1}
		\frac{x_{k+1}-x_k}{\alpha_k} = v_{k+1}-x_{k+1},\\
		\gamma_k 	\frac{v_{k+1}-v_k}{\alpha_k} \in 
		\mu(x_{k+1}-v_{k+1})-\nabla h(x_k)
		-\partial g(x_{k+1})+\xi_{k},
		\label{eq:semi-ADI-vk1}
	\end{subnumcases}
	where the equation \cref{eq:gamma} of the scaling factor 
	$\gamma$ is still discretized implicitly by \cref{eq:gammak}. 
	\subsection{Perturbed gradient mapping}
	Unlike the case \eqref{eq:AIPPA-xk1}, the step \eqref{eq:semi-ADI-vk1} involves inexact calculations of both $\nabla h(x_k)$ and $\proxi_{\lambda g}$. We can simply replace 
	$\nabla h(x_k)$ with inexact data $\widehat{\nabla} h(x_k)$ \cite{attouch_fast_convergence_2018,daspremont_smooth_2008}. This really occurs if one applies Tikhonov regularization \cite{attouch_inertial_2018}. Approximations
	to $\proxi_{\lambda g}$ mainly consider three types of inexact proximal mappings introduced in \cref{sec:inexact-proxi}; see those methods in \cite{monteiro_convergence_2010,villa_accelerated_2013}. 
	
	Following \cite{aujol_stability_2015,schmidt_convergence_2011}, we consider $type$-1 approximation to $\proxi_{\lambda g}$ together with inexact data $\widehat{\nabla } h(x_k)$. Note that in \cite{aujol_stability_2015} only convex case is considered and that in \cite[Propsitions 2 and 4]{schmidt_convergence_2011}, the convergence rates are established with accumulated errors (also see our \cref{lem:conv-Lk}) and no further decay rate is given with specific error, saying $\varepsilon_{k}=O(1/k^p)$. We shall analyse both convex and strongly convex cases and present some new convergence rate estimates.
	
	Motivated by \cite{luo_chen_differential_2019,nesterov_gradient_2013}, for ease of analysis, we introduce the concept of perturbed gradient mapping.
	Given $\lambda>0,\,\tau\geqslant 0$ and $\varepsilon\geqslant 0$, define
	\begin{equation}
		\label{eq:in-gd-map}
		\mathcal G_{\lambda f}(x,\tau,\varepsilon)
		:={}\frac{1}{\lambda}
		(x-{S}_{\lambda f}(x,\tau,\varepsilon))\quad \text{ for all } x\in \mathcal H,
	\end{equation}
	where ${S}_{\lambda f}(x,\tau,\varepsilon)\approx_{1,\varepsilon}\!\proxi_{\lambda g}(x-\lambda\widehat{\nabla}h(x))$ and $\widehat{\nabla}h(x)$ provides some approximation to $\nabla h(x)$ such that $\|\widehat{\nabla}h(x)-\nabla h(x)\|\leqslant \tau/\lambda$. For the case $\varepsilon=0$, we have ${S}_{\lambda f}(x,\tau,0) = \proxi_{\lambda g}(x-\lambda\widehat{\nabla}h(x))$ and by \cref{eq:sub-inf} we see that
	\[
	\mathcal G_{\lambda f}(x,\tau,0)\in \widehat{\nabla}h(x)
	+\partial g\big({S}_{\lambda f}(x,\tau,0)\big).
	\]
	
	Moreover, the following lemma is also crucial for our convergence analysis. Thanks to \cref{lem:key-type1}, the proof is the same as that of \cite[Lemma 3]{luo_chen_differential_2019} and we omit it here.
	\begin{lem}\label{lem:key-Gf}
		Let $\lambda>0,\,\tau\geqslant 0 $ and $\varepsilon\geqslant 0$ be given. 
		Assume $f = h+g$ where $h\in\mathcal S_{\mu,L}^{1,1}$ with $0\leqslant \mu\leqslant L<\infty$ and $g\in\mathcal S_0^0$. For any $x,y\in \mathcal H$, we have 
		\[
		\begin{split}
			\frac{\varepsilon^2}{2\lambda}+	f(y)\geqslant{}&
			f\big({S}_{\lambda f}(x,\tau,\varepsilon)\big)
			+\dual{\mathcal G_{\lambda f}(x,\tau,\varepsilon),y-x}
			+\frac{1}{\lambda}\dual{\sigma+e,y-S_{\lambda f}(x,\tau,\varepsilon)}\\
			{}&\qquad+\frac{\mu}{2}\nm{y-x}^2
			+\frac{\lambda}{2}(2-\lambda   L)
			\nm{\mathcal G_{\lambda f}(x,\tau,\varepsilon)}^2,
		\end{split}
		\]
		where $\sigma = {S}_{\lambda f}(x,\tau,\varepsilon)-{S}_{\lambda f}(x,\tau,0)$ 
		and $e =\lambda\big(\nabla h(x)- \widehat{\nabla}h(x)\big)$.
	\end{lem}
	\subsection{The proposed algorithm}
	Let $\{\varepsilon_k\}_{k=0}^\infty$ and $\{\tau_k\}_{k=0}^\infty$ be two nonnegative sequences. We make the step \eqref{eq:semi-ADI-xk1} semi-implicit and plug the gradient mapping \cref{eq:in-gd-map} into \cref{eq:semi-ADI-vk1} to obtain 
	\begin{subnumcases}{}
		\label{eq:AIPGM-mid-vk1}
		\frac{x_{k+1}-x_k}{\alpha_k}=v_{k}-x_{k+1},\\
		\frac{v_{k+1}-v_{k}}{\alpha_k}={}
		\frac{\mu}{\gamma_k}  (x_{k+1}-v_{k+1})
		-\frac{1}{\gamma_k}
		\mathcal G_{\lambda_k f}(x_{k+1},\tau_k,\varepsilon_k),
		\label{eq:AIPGM-mid-xk1}
	\end{subnumcases}
	where $\lambda_k>0$ and $\widehat{\nabla}h(x_k)\approx\nabla h(x_k)$ is some approximation such that $\|\widehat{\nabla}h(x_k)-\nabla h(x_k)\|\leqslant \tau_k/\lambda_k$.
	
	Under the exact case $\tau_k =\varepsilon_k=0$, the scheme \cref{eq:AIPGM-mid-vk1} has been considered in \cite[Section 7.2]{luo_chen_differential_2019}, where, to promise the decay property of a Lyapunov function, an additional gradient descent step is supplemented. Hence, we also add an extra gradient step to \cref{eq:AIPGM-mid-vk1} and obtain the following scheme:
	\begin{subnumcases}{}
		\frac{y_k-x_{k}}{\alpha_k}={}v_{k}-y_k,
		\label{eq:AIPGM-yk}\\
		\frac{v_{k+1}-v_{k}}{\alpha_k}={}
		\frac{\mu}{\gamma_k}(y_k-v_{k+1})
		-\frac{1}{\gamma_k}
		\mathcal G_{\lambda_k f}(y_{k},\tau_k,\varepsilon_k),
		\label{eq:AIPGM-vk1}\\
		x_{k+1} ={}y_k-\lambda_k\mathcal G_{\lambda_k f}(y_{k},\tau_k,\varepsilon_k).
		\label{eq:AIPGM-xk1}
	\end{subnumcases}
	As mentioned 
	before, equation \cref{eq:gamma} for $\gamma$ is discretized 
	implicitly by \cref{eq:gammak}. Besides, we  impose the condition 
	\begin{equation}\label{eq:cond-step}
		\lambda_k = 1/L,\quad 2L\alpha_k^2=\gamma_{k}(\alpha_k+1).
	\end{equation}
	
	Note that both $\{\alpha_k\}_{k=0}^{\infty}$ and $\{\gamma_k\}_{k=0}^{\infty}$ are well defined by \cref{eq:gammak,eq:cond-step} and 
	\[
	\begin{split}
		&\min\{\gamma_0,\mu\}
		\leqslant \gamma_k\leqslant \max\{\gamma_0,\mu\},\quad 
		\min\{\alpha_{0},\alpha_{\mu}\}
		\leqslant \alpha_k\leqslant \max\{\alpha_{0},\alpha_{\mu}\},
	\end{split}
	\]
	for all $k\in\mathbb N$,
	where $\alpha_\mu$ solves $2L\alpha_\mu^2=\mu(1+\alpha_{\mu})$, i.e.,
	\begin{equation}\label{eq:amu}
		\alpha_{\mu} = \frac{1}{4L}\left(\mu+\sqrt{\mu^2+8\mu L}\right)
		\in\left[\sqrt{\frac{\mu}{2L}},
		\sqrt{\frac{\mu}{L}}\,
		\right].
	\end{equation}
	Moreover, as $k\to\infty$, we have $\alpha_k\to\alpha_\mu$ 
	and $\gamma_k\to\mu$ and particularly, if $\gamma_0=\mu$, then $\gamma_k=\mu$ and $\alpha_k = \alpha_{\mu}$; see \cref{lem:gammak-ak} for more detailed asymptotic estimates.
	
	We now summarize the scheme \cref{eq:AIPGM-yk} together 
	with \cref{eq:gammak,eq:cond-step} as an algorithm below.
	\begin{small}
		\begin{algorithm}[H]
			\caption{Inexact Accelerated Proximal Gradient Method}
			\label{algo:AIPGM}
			\begin{algorithmic}[1] 
				\REQUIRE  $x_0,v_0\in \mathcal H,\gamma_0>0$ and $\lambda=1/L$.
				\FOR{$k=0,1,\ldots$}
				\STATE Set $\Delta_k = \gamma_k^2+4L\gamma_k$ and compute $\alpha_k=\lambda(\gamma_k+\sqrt{\Delta_k})/2$. 
				\STATE Update $\gamma_{k+1} = (\gamma_k+\mu\alpha_k)/(1+\alpha_k)$.
				\STATE Set $y_k = (x_k+\alpha_kv_k)/(1+\alpha_k)$.
				\STATE Given $\tau_k\geqslant 0$, compute $\widehat{\nabla }h(y_k)$ such that $\|\nabla h(y_k)-\widehat{\nabla }h(y_k)\|\leqslant L \tau_k$.
				\STATE Given $\varepsilon_k\geqslant 0$, compute $x_{k+1} =_{1,\varepsilon_k}\!\proxi_{\lambda g}(y_k-\lambda\widehat{\nabla}h(x_k))$.
				\STATE Update $\displaystyle v_{k+1} = 
				\frac{1}{\gamma_{k}+\mu\alpha_k}
				\left(			\gamma_kv_k+\mu\alpha_ky_k-L\alpha_k(y_k-x_{k+1})\right)$.	
				\ENDFOR
			\end{algorithmic}
		\end{algorithm}
	\end{small}
	\subsection{Rate of convergence}
	Next, let us carefully investigate the convergence behaviour
	of \cref{algo:AIPGM} under the error decay assumption:
	\begin{equation}\label{eq:p-q}
		\varepsilon_k\leqslant \frac{1}{(k+1)^p}
		\quad\text{ and }\quad 
		\tau_k\leqslant\frac{1}{(k+1)^q},\quad p,\,q>0.
	\end{equation}
	The following two theorems give convergence rates for convex case ($\mu=0$) and strongly convex case ($\mu>0$), respectively. Our \cref{thm:conv-AIPGM-mu-0} recovers the results in
	\cite[Corrolaries 3.6 and 3.7]{aujol_stability_2015}, and \cref{thm:conv-AIPGM-mu>0} yields new convergence rate estimate.
	\begin{thm}
		\label{thm:conv-AIPGM-mu-0}
		Assume $f = h+g$ where $h\in\mathcal S_{0,L}^{1,1}$ and $g\in\mathcal S_0^0$. If $\gamma_0=L$ and \cref{eq:p-q} is true with $p,\,q>1$, then for \cref{algo:AIPGM}, we have 
		\begin{equation}
			\label{eq:conv-AIPGM-mu-0}
			f(x_k)-f(x^*)\leqslant 
			\frac{2\mathcal L_0}{(k+1)^2} 
			+ L\big(C_pT_k(p)+C_qT_k(q)\big),
		\end{equation}
		where $\mathcal L_0: = 
		f(x_{0})-f(x^*) + 
		\frac{\gamma_{0}}{2}
		\nm{v_{0}-x^*}^2$, and for any $r>1$, $T_k(r)$ is defined by that
		\[
		T_k(r):=\left\{
		\begin{aligned}
			&\frac{1+\ln^2(k+1)}{(k+1)^{2}},&&r=2,\\
			&(k+1)^{-2}+(k+1)^{2-2r},&&r>1 \text{ and } r\neq 2.
		\end{aligned}
		\right.
		\]
	\end{thm}
	\begin{rem}
		In the setting of \cref{thm:conv-AIPGM-mu-0} (except \cref{eq:p-q}), if $\{k(\varepsilon_k+\tau_k)\}_{k=0}^\infty$ is summable, then we can obtain 
		the rate $O(1/k^2)$ without log factor, which is promised by the estimate \cref{eq:bd-Lk-mu-0}. This also agrees with the result in \cite{attouch_fast_convergence_2018}, where the case $\varepsilon_k = 0$ has been considered.
	\end{rem}
	\begin{thm}
		\label{thm:conv-AIPGM-mu>0}
		Assume $f = h+g$ where $h\in\mathcal S_{\mu,L}^{1,1}$ 
		with $\mu>0$ and $g\in\mathcal S_0^0$.
		If $\gamma_0 = \mu$ and \cref{eq:p-q} holds true, 
		then for \cref{algo:AIPGM}, we have
		\begin{equation}\label{eq:conv-AIPGM-mu>0}
			\small
			f(x_k)-f(x^*)+\frac{\mu}{2}\nm{v_k-x^*}^2
			\leqslant 
			\frac{2\mathcal L_0}{(1+\sqrt{0.5\mu/L})^k} 
			+L\left(\frac{L}{\mu}\right)^2\left(\frac{C_{p}}{(k+1)^{2p}}
			+\frac{C_{q}}{(k+1)^{2q}}\right),
		\end{equation}	
		where $\mathcal L_0: = 
		f(x_{0})-f(x^*) + 
		\frac{\gamma_{0}}{2}
		\nm{v_{0}-x^*}^2$.	
	\end{thm}
	\begin{rem}\label{rem:err-exp}
		Under the assumption of \cref{thm:conv-AIPGM-mu>0} (except \cref{eq:p-q}), if $	\sum_{k=0}^{\infty}(1+\alpha_\mu)^{k/2}(\varepsilon_k+\tau_k)<\infty$,
		then by \cref{eq:bd-mu>0}, we can recover the accelerated linear rate $O((1+\sqrt{0.5\mu/L})^{-k})$.
	\end{rem}
	In the sequel, we shall use the sequence $\{\beta_k\}_{k=0}^{\infty}$ 
	define by \cref{eq:betak}. With
	the choice \cref{eq:cond-step}, the asymptotic behaviour 
	of $\beta_k$ has been given in \cref{lem:gammak-ak}.
	Let 
	\begin{equation}\label{eq:xik}
		\xi_k =\sigma_k+e_k,
	\end{equation}
	where 
	$e_k = \lambda_k\big(\nabla h(y_k)-\widehat{\nabla }h(y_k)\big)$ satisfies $\nm{e_k}\leqslant  \tau_k$ and
	\[
	\sigma_k = x_{k+1}-\proxi_{\lambda_k g}(x_k-\lambda_k\widehat{\nabla}h(x_k)),
	\]
	which, by \cref{lem:key-type1}, admits the estimate $\nm{\sigma_k}
	\leqslant \varepsilon_k$. Hence, we get 
	\begin{equation}\label{eq:est-xik}
		\nm{\xi_k}\leqslant \varepsilon_k+\tau_k.
	\end{equation}
	In addition, except for the Lyapunov function \cref{eq:Lk}, we define
	\[
	\left\{
	\begin{split}
		{}&\Upsilon_0 := 0,\,\Upsilon_k:={}L
		\sum_{i=0}^{k-1}\frac{2}{\beta_{i+1}}
		(\varepsilon_i^2+\tau_i^2), \quad k\geqslant 1,\\
		{}&	\Omega_0={}0,\quad
		\Omega_k = L\sum_{i=0}^{k-1}\frac{\alpha_i}{\sqrt{\beta_{i}\gamma_{i}}}(\varepsilon_i+\tau_i),
		\quad k\geqslant 1.
	\end{split}
	\right.
	\]
	\begin{lem}\label{lem:conv-Lk}
		For \cref{algo:AIPGM}, we have 
		\begin{equation}\label{eq:conv-Lk}
			\mathcal L_k\leqslant 2\beta_k
			\left(\mathcal L_0+\Upsilon_k+
			\Omega_k^2
			\right).
		\end{equation}
	\end{lem}
	\begin{proof}
		Let us first establish
		\begin{equation}\label{eq:1st-step}
			\mathcal L_{k+1}-\mathcal L_{k}
			\leqslant {}-\alpha_k\mathcal L_{k+1}+\frac{\alpha_k}{\lambda_k}
			\nm{\xi_k}	\nm{v_{k}-x^*}
			+	\frac{2}{\lambda_k}
			(1+\alpha_k)(\varepsilon_k^2
			+\tau_k^2).
		\end{equation}
		
		Thanks to the formulation \cref{eq:AIPGM-xk1}, the estimate 
		of $	\mathcal L_{k+1}-	\mathcal L_{k} $ is more or less parallel to that of 
		\cite[Lemma 3]{luo_chen_differential_2019}. Following the proof, we can bound the difference by that
		\[
		\begin{split}
			\mathcal L_{k+1}-\mathcal L_{k}
			\leqslant {}& f(x_{k+1})-f(x_k)-
			\frac{\alpha_k\gamma_{k+1}}{2}
			\nm{v_{k+1}-x^*}^2\\
			{}& \quad+\frac{\alpha_k^2}{2\gamma_k}\nm{\mathcal G_{\lambda_k f}(y_k,\tau_k,\varepsilon_k)}^2- \dual{ \mathcal G_{\lambda_k f}(y_k,\tau_k,\varepsilon_k), y_{k} - x_{k}} \\
			{}&\qquad +\frac{\mu\alpha_k}{2}
			\nm{y_{k}-x^{*}}^2- \alpha_k\dual{  \mathcal G_{\lambda_k f}(y_k,\tau_k,\varepsilon_k), y_{k} - x^{*}}.
		\end{split}
		\]
		Invoking \cref{lem:key-Gf}, we have
		\[
		\begin{split}
			f(x_{k+1})-f(x_k) 
			\leqslant{}&\frac{\varepsilon_k^2}{2\lambda_k}+
			\dual{ \mathcal G_{\lambda_k f}(y_k,\tau_k,\varepsilon_k), y_{k} - x_{k}}
			+\frac{1}{\lambda_k}\dual{ \xi_k,x_{k+1}-x_k}
			\\
			{}&\qquad	-\frac{\mu}{2}\nm{y_k-x_k}^2-\frac{\lambda_k}{2}
			\nm{\mathcal G_{\lambda_k f}(y_k,\tau_k,\varepsilon_k)}^2,\\
			f(x_{k+1})-f(x^*)
			\leqslant{}&\frac{\varepsilon_k^2}{2\lambda_k}
			+\dual{ \mathcal G_{\lambda_k f}(y_k,\tau_k,\varepsilon_k), y_{k} - x^*} +\frac{1}{\lambda_k}\dual{ \xi_k,x_{k+1}-x^*}\\
			{}&\quad
			-\frac{\mu}{2}\nm{y_k-x^*}^2
			- \frac{\lambda_k}{2}
			\nm{\mathcal G_{\lambda_k f}(y_k,\tau_k,\varepsilon_k)}^2,
		\end{split}
		\]
		where $\xi_k$ is defined by \cref{eq:xik}. Hence, 
		dropping surplus negative square  terms yields that
		\[
		\begin{split}
			\mathcal L_{k+1}-\mathcal L_{k}
			\leqslant {}&-\alpha_k\mathcal L_{k+1}+
			\frac{1+\alpha_k}{2\lambda_k}\varepsilon_k^2
			+\frac{1}{\lambda_k}\dual{ \xi_k,x_{k+1}-x_k}
			+\frac{\alpha_k}{\lambda_k}\dual{ \xi_k,x_{k+1}-x^*}\\
			{}&\quad+\frac{\alpha_k^2}{2\gamma_k}\nm{\mathcal G_{\lambda_k f}(y_k,\tau_k,\varepsilon_k)}^2-\frac{(1+\alpha_k)\lambda_k}{2}
			\nm{\mathcal G_{\lambda_k f}(y_k,\tau_k,\varepsilon_k)}^2.
		\end{split}
		\]
		Recalling the relation $x_{k+1}-y_k = \lambda_k \mathcal G_{\lambda_k f}(y_k,\tau_k,\varepsilon_k)$, the cross terms are bounded as follows
		\[
		\begin{split}
			{}&	\frac{1}{\lambda_k}\dual{ \xi_k,x_{k+1}-x_k}
			+\frac{\alpha_k}{\lambda_k}\dual{ \xi_k,x_{k+1}-x^*}\\
			={}&\frac{1}{\lambda_k}\dual{ \xi_k,y_{k}-x_k}
			+\frac{\alpha_k}{\lambda_k}\dual{ \xi_k,y_{k}-x^*}
			+(1+\alpha_k)\dual{\xi_k,\mathcal G_{\lambda_k f}(y_k,\tau_k,\varepsilon_k)}\\
			\leqslant {}&\frac{1}{\lambda_k}\dual{ \xi_k,y_{k}-x_k}
			+\frac{\alpha_k}{\lambda_k}\dual{ \xi_k,y_{k}-x^*}
			+(1+\alpha_k)\left(\frac{\lambda_k}{4}
			\nm{\mathcal G_{\lambda_k f}(y_k,\tau_k,\varepsilon_k)}^2
			+\frac{\nm{\xi_k}^2}{\lambda_k}\right)\\	
			= {}&\frac{\alpha_k}{\lambda_k}\dual{ \xi_k,v_{k}-x^*}
			+\frac{\lambda_k}{4}(1+\alpha_k)
			\nm{\mathcal G_{\lambda_k f}(y_k,\tau_k,\varepsilon_k)}^2
			+(1+\alpha_k)\frac{\nm{\xi_k}^2}{\lambda_k},
		\end{split}
		\]
		where in the last step, we used \cref{eq:AIPGM-yk}. Consequently, by \cref{eq:cond-step,eq:est-xik}, we get
		\[
		\begin{split}
			\mathcal L_{k+1}-\mathcal L_{k}
			\leqslant {}&-\alpha_k\mathcal L_{k+1}+\frac{\alpha_k}{\lambda_k}
			\dual{ \xi_k,v_{k}-x^*}+
			\frac{1+\alpha_k}{\lambda_k}(\varepsilon_k^2+\nm{\xi_k}^2)\\
			{}&\quad+\frac{1}{4\gamma_k}
			\left(2\alpha_k^2
			-\lambda_k\gamma_k(\alpha_k+1)
			\right)
			\nm{\mathcal G_{\lambda_k f}(y_k,\tau_k,\varepsilon_k)}^2\\
			\leqslant {}&-\alpha_k\mathcal L_{k+1}+\frac{\alpha_k}{\lambda_k}
			\nm{\xi_k}\nm{v_{k}-x^*}
			+	\frac{2}{\lambda_k}
			(1+\alpha_k)(\varepsilon_k^2
			+\tau_k^2),
		\end{split}
		\]
		which proves \cref{eq:1st-step}.
		
		Now, using \cref{lem:Gronwall-discrete} and adopting the proof of \cref{eq:est-cross}, 
		we can obtain \cref{eq:conv-Lk} and thus conclude the proof of this lemma.
	\end{proof}
	\medskip\noindent{\bf Proof of \cref{thm:conv-AIPGM-mu-0}.}
	Since $\mu=0$ and $\gamma_0=L$, by \cref{lem:gammak-ak}, we have 
	\[
	\frac{2}{(k + \sqrt{2})^2}
	\leqslant \beta_k\leqslant\frac{8}{(k + 2\sqrt{2})^2}\quad\text{and}\quad
	\frac{L\alpha_k}{\sqrt{\beta_k\gamma_k}}\leqslant \sqrt{2L}(k+\sqrt{2}).
	\]
	Therefore, it follows that
	\[
	\Upsilon_k
	\leqslant{} 4L\sum_{i=0}^{k-1} \big(i+1\big)^2\left(\varepsilon_i^2+\tau_i^2\right)
	\quad\text{and}\quad
	\Omega_k\leqslant {}2\sqrt{L}\sum_{i=0}^{k-1} \big(i+1\big)\left(\varepsilon_i+\tau_i\right),
	\]
	and by \cref{lem:conv-Lk} we see the bound
	\begin{equation}\label{eq:bd-Lk-mu-0}
		\small
		\mathcal L_k\leqslant 
		\frac{16}{(k + 2\sqrt{2})^2}
		\left[\mathcal L_0+
		4L\sum_{i=0}^{k-1} \big(i+1\big)^2\left(\varepsilon_i^2+\tau_i^2\right)
		+4L\left(\sum_{i=0}^{k-1} \big(i+1\big)\left(\varepsilon_i+\tau_i\right)\right)^2
		\right].
	\end{equation}
	According to the decay assumption \cref{eq:p-q} and the previous two estimates \cref{eq:1-p,eq:2-2p}, we obtain the desired result \cref{eq:conv-AIPGM-mu-0} from \cref{eq:bd-Lk-mu-0} and thus 
	finish the proof of \cref{thm:conv-AIPGM-mu-0}. \qed
	\vskip0.1cm
	\medskip\noindent{\bf Proof of \cref{thm:conv-AIPGM-mu>0}.} 
	As $\gamma_0=\mu>0$, we have
	\begin{equation}\label{eq:bd-mu>0}
		\Upsilon_k
		\leqslant{} 2L\sum_{i=1}^{k} \big(1+\alpha_{\mu}\big)^{i}\left(\varepsilon_{i+1}^2+\tau_{i+1}^2\right)
		\quad\text{and}\quad
		\Omega_k\leqslant {}\sqrt{L}\sum_{i=0}^{k-1} \big(1+\alpha_{\mu}\big)^{i/2}\left(\varepsilon_i+\tau_i\right),
	\end{equation}
	By \cref{lem:estk-It}, it holds that
	\[
	\begin{split}
		\sum_{i=1}^{k} (1+\alpha_\mu)^{i}\varepsilon_{i+1}^2\leqslant{}&
		\sum_{i=1}^{k} \frac{(1+\alpha_\mu)^{i}}{(i+1)^{2p}}
		\leqslant  \frac{C_p(1+\ln(1+\alpha_\mu))}{\snm{\ln(1+\alpha_\mu)}^2}
		\cdot\frac{(1+\alpha_\mu)^{k+1}}{(k+1)^{2p}},
	\end{split}
	\]
	and then we use the fact \cref{eq:amu} and the trivial estimate $0.5\alpha_\mu\leqslant \ln(1+\alpha_\mu)\leqslant \alpha_{\mu}\leqslant 1$ to get that
	\[
	\begin{split}
		\sum_{i=1}^{k} (1+\alpha_\mu)^{i}\varepsilon_{i+1}^2\leqslant{}&
		C_{p}(1+\alpha_{\mu})^{k}\cdot\frac{L/\mu}{(k+1)^{2p}}.
	\end{split}
	\]
	An anlaougous argument implies 
	\[
	\sum_{i=0}^{k-1} (1+\alpha_\mu)^{i/2}\varepsilon_{i}\leqslant{}
	C_{p}(1+\alpha_{\mu})^{k/2}\cdot\frac{L/\mu}{(k+1)^{p}},
	\]
	and similar bounds related to $\tau_i$ can be obtained.
	Combining these estimates with \cref{lem:conv-Lk} proves 
	\cref{eq:conv-AIPGM-mu>0} and concludes the proof of 
	\cref{thm:conv-AIPGM-mu>0}. \qed
	\section{Numerical Results}
	\label{sec:num}
	In this part, we conduct two numerical experiments to illustrate the performance of our \cref{algo:AIPGM} with perturbed gradients. 
	Namely, we take
	\begin{equation}\label{eq:err}
		\varepsilon_k = 0,\quad \tau_k = \frac{\tau}{(k+1)^p},\quad \tau>0,\,0<p\leqslant \infty,
	\end{equation}
	where $p = \infty$ means $\tau_ k= 0$. To verify the claim in \cref{rem:err-exp} for strongly convex case, we also consider 
	\begin{equation}\label{eq:err-exp}
		\varepsilon_k = 0,\quad \tau_k = \frac{(k+1)^{-p}}{(1+\alpha_\mu)^{k/2}},\quad 1<p\leqslant \infty,
	\end{equation}
	where $\alpha_\mu$ is defined by \cref{eq:amu}.
	We will report the results of a quadratic programming and the Lasso problem. For both two cases, we run \cref{algo:AIPGM} with enough iterations to obtain a convincible minimum $f^*$.
	\vskip0.1cm
	\noindent$\bullet${\,\bf Quadratic programming.}
	Consider 
	\begin{equation}\label{eq:QP}
		\min_{x\in \R^n}\,\frac{1}{2}x^{\top}Ax - b^\top x\quad {\rm s.t.~}l\leqslant x\leqslant u,
	\end{equation}
	where $b,\,l,\,u\in\R^n$ and $A\in\R^{n\times n}$ is symmetric positive semidefinite. We first choose $A = Q^{\top}Q$, where $Q\in\R^{n\times n}$ is generated from the uniformly distribution on $(0,1)$. In \cref{algo:AIPGM}, we take $\mu=0$ and estimate $L$ by the sum of the diagonal components of $A$. Numerical outputs with $n = 400$ and $n = 800$ are plotted in \cref{fig:QP-mu0}, from which we observed the {\it local} fast decay rate $O(k^{-\min\{2,2p\}})$, even for $p<1$. But small $p$ ($\leqslant 0.5$) does not promise convergence (one may expect convergence for sufficient large steps). However, our \cref{thm:conv-AIPGM-mu-0} only gives the  global rate $O(k^{-\min\{2,2p-2\}})$ for $p>1$, which is pessimistic compared with the numerical results.
	\begin{figure}[H]
		\centering
		\includegraphics[width=420pt,height=180pt]{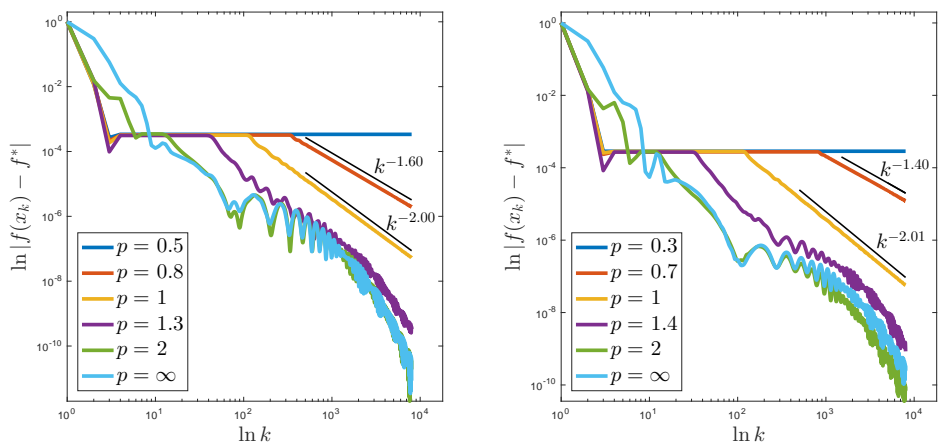}
		\caption{Convergence rate of \cref{algo:AIPGM} for problem \cref{eq:QP} with perturbation \cref{eq:err} ($\tau=1$).}
		\label{fig:QP-mu0}
	\end{figure}
	We then consider a sparse and symmetric positive definite matrix $A$ with size $n = 1089$. It is obtained from the finite element discretization of the Poisson equation on $[0,1]^2$ using piecewise continuous linear polynomials. For \cref{algo:AIPGM}, we set $\mu=\lambda_{\min}(A)$ and $L = \lambda_{\max}(A)$. We consider two kinds of perturbations \cref{eq:err} ($\tau=1$) and \cref{eq:err-exp} and report the results in \cref{fig:QP-mu}. For those two cases, the algebraic rate $O(k^{-2p})$ and the linear rate are observed respectively, which agree with \cref{thm:conv-AIPGM-mu>0,rem:err-exp}.
	\begin{figure}[H]
		\centering
		\includegraphics[width=420pt,height=180pt]{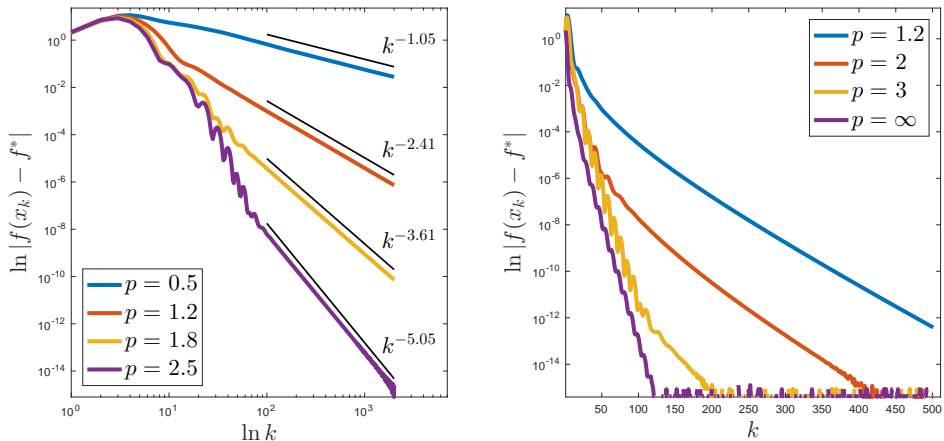}
		\caption{Convergence behaviour of \cref{algo:AIPGM} for \cref{eq:QP} with perturbations \cref{eq:err} (left) and \cref{eq:err-exp} (right).}
		\label{fig:QP-mu}
	\end{figure}
	\vskip0.1cm
	\noindent$\bullet${\,\bf Lasso.} We then focus on the well-known least absolute shrinkage and selection operator (Lasso) problem 
	\begin{equation}\label{eq:Lasso}
		\min_{x\in \R^n}\,\frac{1}{2}\nm{Ax-b}^2 +\rho \nm{x}_{l^1},
	\end{equation}
	where $\rho>0,\,b\in\R^{m},\,A\in\R^{m\times n}$ and $m\ll n$. We generate $A$ from standard normal distribution and take $b = Ay+e$, where $e$ denotes the noise and $y$ is sparse with $s$ nonzero components. The parameter is chosen as $\rho = 0.5$. For \cref{algo:AIPGM}, we set $\mu=0$ and take $L$ as the sum of the diagonal components of $A^{\top}A$. Even though the perturbation \cref{eq:err} is decreasing, the errors in the few initial steps are large. Therefore, to observe the local convergence rate, we set $\tau = 1e$-2.
	
	From numerical results in \cref{fig:Lasso}, we find that (i) similar with the previous example (cf. \cref{fig:QP-mu0}), small $p$ ($<0.5$) cannot promise convergence, and (ii) after a large number of iterations, the perturbations become small and the local decay rate $O(k^{-2})$ arises. This together with the previous test for quadratic programming \cref{eq:QP} shows the fast local convergence rate which is better than the global one $O(k^{-\min\{2,2p-2\}})$ given by \cref{thm:conv-AIPGM-mu-0}. 
	\begin{figure}[H]
		\centering
		\includegraphics[width=430pt,height=150pt]{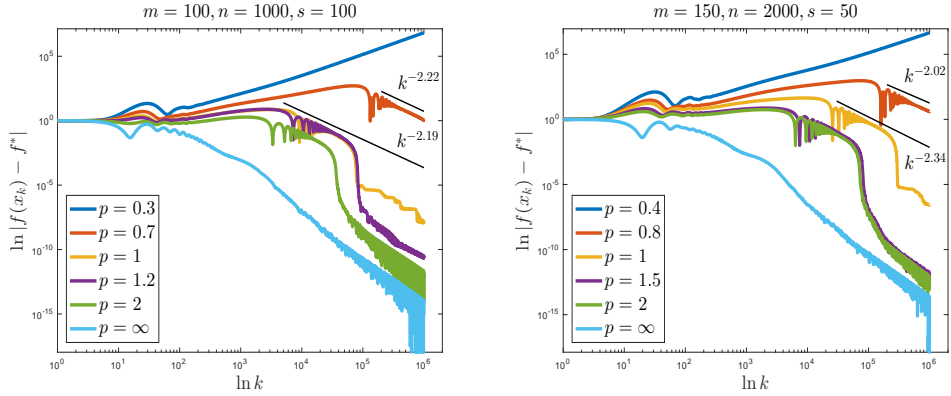}
		\caption{Convergence behaviour of \cref{algo:AIPGM} for the Lasso problem \cref{eq:Lasso} with perturbation \cref{eq:err} ($\tau = 1e$-2).}
		\label{fig:Lasso}
	\end{figure}
	
	\appendix
	\section{Some Auxiliary Estimates}
	In this section, we shall present some basic estimates. 
	We first cite an important Gronwall-type inequality; 
	see \cite[Proposition 1.2]{barbu_differential_2016}.
	\begin{lem}[\cite{barbu_differential_2016}]
		\label{lem:Gronwall}
		Let $T>0$ and assume that $y\in C[0,T]$ and $w\in L^1(0,T)$.
		If 
		\[
		\frac{1}{2}y^2(t)\leqslant 
		\frac{1}{2}c^2+\int_{0}^{t}w(s)y(s)
		\dd s,
		\quad 0<t\leqslant T,
		\]
		where $c\in\mathbb R$, then 
		\[
		\snm{y(t)}\leqslant 	\snm{c}+\int_{0}^{t}\snm{w(s)}\mathrm ds,
		\quad 0<t\leqslant T.
		\]	
	\end{lem}
	We also need a discrete version of \cref{lem:Gronwall}, 
	which can be found in \cite[Lemma 5.1]{attouch_fast_convergence_2018} 
	or \cite[Lemma A.4]{aujol_stability_2015}.
	\begin{lem}[\cite{attouch_fast_convergence_2018}]
		\label{lem:Gronwall-discrete}
		Let $\{a_{k}\}_{k=0}^\infty ,\,\{b_{k}\}_{k=0}^\infty $ and $\{c_{k}\}_{k=0}^\infty $ be three real sequences. If $\{c_{k}\}_{k=0}^\infty $ 
		is nondecreasing and for all $k\in\mathbb N$,
		\begin{equation*}\label{eq:discrete-Gronwall}
			a^2_{k}\leqslant	c_k^2 + \sum_{i=0}^{j_k}a_ib_i,
			\quad\text{ for some } 0\leqslant j_k\leqslant k,
		\end{equation*}
		then it holds that
		\begin{equation*}
			\snm{a_{k}} \leqslant\snm{c_k}+\sum_{i=0}^{j_k}\snm{b_i},
			\quad k\in\mathbb N.
		\end{equation*}
	\end{lem}
	
	Next,we present an estimate.
	\begin{lem}\label{lem:est-It}
		Assume that $p\geqslant 0$ and $A>1$, then for all $t>0$,
		\begin{equation}\label{eq:est-It}
			\int_0^t\frac{A^{s}}{(s+1)^p}\dd s
			\leqslant 
			\frac{C_p(1+\ln A)}{\snm{\ln A}^2}\cdot
			\frac{A^{t}}{(t+1)^p},
		\end{equation}
		where $C_p>0$ depends only on $p$.
	\end{lem}
	\begin{proof}
		It is trivial to verify \cref{eq:est-It} with $p=0$, so we consider $p>0$.
		Define
		\[
		I(t): = 	\int_0^t\frac{A^{s+1}}{(s+1)^p}\dd s,\quad t\geqslant0.
		\]
		Observing the infinite series expansion
		\[
		A^{s+1} = e^{(s+1)\ln A} = 	\sum_{n=0}^\infty \frac{((s+1)\ln A)^{n}}{n!},
		\]
		which is uniformly convergent over $[0,t]$, we can exchange the order of integral and summation to obtain that
		\begin{equation}
			\label{eq:It}
			I(t) = {}
			\int_{0}^{t}(s+1)^{-p}
			\sum_{n=0}^\infty \frac{((s+1)\ln A)^{n}}{n!}
			\dd s
			= \sum_{n=0}^\infty
			\frac{(\ln A)^{n}}{n!}
			\int_{0}^{t}(s+1)^{n-p}
			\dd s.
		\end{equation}
		
		Consider first that $p\notin\mathbb N$, then
		$n+1-p\neq0$ for all $n\in\mathbb N$. 
		By \cref{eq:It}, we have $	I(t)= I_1+I_2$, where 
		\[
		\left\{
		\begin{aligned}
			I_1:={}&
			\sum_{n=0}^{[p]-1}
			\frac{(\ln A)^n}{n!}\cdot
			\frac{(t+1)^{n+1-p}-1}{n+1-p},\\
			I_2:={}&
			\sum_{n=[p]}^{\infty}
			\frac{(\ln A)^n}{n!}\cdot
			\frac{(t+1)^{n+1-p}-1}{n+1-p},
		\end{aligned}	
		\right.
		\]
		where as usual $[p]$ denotes the integer part of $p$.
		Clearly, $I_1$ is nonpositive for $p>1$ and vanishes for $0<p<1$,
		and the second term $I_2$ is bounded by that
		\begin{equation}\label{eq:I2}
			\small
			I_2= {}\frac{1}{(t+1)^p\ln A}
			\sum_{n=[p]}^{\infty}
			\frac{((t+1)\ln A)^{n+1}}{(n+1)!}
			\cdot\frac{n+1}{n+1-p}
			\leqslant {}
			\frac{1+[p]}{(1+[p]-p)\ln A}
			\cdot
			\frac{A^{t+1}}{(t+1)^p}.
		\end{equation}
		Hence, we obtain
		\begin{equation}\label{eq:I-2}
			I(t)	\leqslant {}
			\frac{1+[p]}{(1+[p]-p)\ln A}
			\cdot
			\frac{A^{t+1}}{(t+1)^p}
		\end{equation}
		
		Then, we consider $p\in\mathbb N$. 
		From \cref{eq:It} we obtain that
		\begin{equation}
			\label{eq:I-3-temp}
			I(t)= I_2+I_3+\frac{(\ln A)^{p-1}}{(p-1)!}\ln (t+1),
		\end{equation}
		where 
		\[
		I_3:={}	\sum_{n=0}^{p-2}
		\frac{(\ln A)^n}{n!}\cdot
		\frac{(t+1)^{n+1-p}-1}{n+1-p}.
		\]
		As $I_3$ is zero for $p=1$ and nonpositive for $p>1$, we conclude from \cref{eq:I2,eq:I-3-temp} that
		\[
		I(t)\leqslant \frac{1+p}{\ln A}
		\cdot\frac{A^{t+1}}{(t+1)^p}
		+\frac{(\ln A)^{p-1}}{(p-1)!}\ln (t+1).
		\]
		Therefore, applying the elementary inequality
		\[
		\ln(t+1)
		\leqslant \frac{t+1}{e}\leqslant
		\frac{(p+1)^{p+1}}{e(e\ln A)^{p+1}}\cdot
		\frac{A^{t+1}}{(t+1)^{p}}\quad \forall\,t\geqslant 0,
		\]
		and thanks to Stirling's formula 
		\[
		\sqrt{2 \pi} (p-1)^{p-\frac{1}{2}} e^{1-p} \leqslant (p-1) !,
		\]
		for all $p\geqslant 2$, we obtain that
		\[
		I(t)\leqslant 
		\frac{C_p(1+\ln A)}{\snm{\ln A}^2}\cdot
		\frac{A^{t+1}}{(t+1)^p}.
		\]
		This together with \cref{eq:I-2} proves \cref{eq:est-It} and completes the proof of this lemma.
	\end{proof}
	A discrete version of \cref{lem:est-It} is given below.
	\begin{coro}\label{lem:estk-It}
		Assume $p\geqslant 0$ and $A>1$, then for all $k\geqslant 1$,
		\begin{equation*}
			\sum_{i=1}^{k}\frac{A^i}{(i+1)^p}\leqslant 
			\frac{C_p(1+\ln A)}{\snm{\ln A}^2}
			\cdot\frac{A^{k+1}}{(k+1)^p},
		\end{equation*}
		where $C_p>0$ depends only on $p$.
	\end{coro}
	\begin{proof}
		Note that
		\[
		\sum_{i=1}^{k}\frac{A^{i}}{(i+1)^p}\leqslant \sum_{i=1}^{k}\int_{i-1}^{i}\frac{A^{s+1}}{(s+1)^p}\dd s
		=\int_{0}^{k}\frac{A^{s+1}}{(s+1)^p}\dd s,
		\]
		and by \cref{lem:est-It} we obtain the desired result.
	\end{proof}
	
	To the end, we list a lemma that depicts the asymptotic behaviour of some 
	key sequences, and for detailed proof we refer to \cite[Lemma B2]{luo_chen_differential_2019}.
	\begin{lem}\label{lem:gammak-ak}
		Given $\gamma_0>0,\,Q>0$ and $q\geqslant 0$, define $\{\alpha_k\}_{k=0}^{\infty}$ 
		and $\{\gamma_k\}_{k=0}^{\infty}$ by that
		\[
		\left\{
		\begin{split}
			Q\alpha_k^2 = {}&\gamma_{k}(1+\alpha_k),\,\alpha_k>0,\\
			\gamma_{k+1} = {}&\gamma_k+\alpha_k(q-\gamma_{k+1}).
		\end{split}
		\right.
		\]
		Then we have the following.
		\begin{itemize}
			\item $\gamma_k>0,\,\min\{\gamma_0,q\}
			\leqslant \gamma_k\leqslant \max\{\gamma_0,q\}$ and $\gamma_k\to q$ as $k\to\infty$. 		
			\item $\min\{\alpha_{0},\alpha_{q}\}
			\leqslant \alpha_k\leqslant \max\{\alpha_{0},\alpha_{q}\}$ and $\alpha_k\to\alpha_q$ as $k\to\infty$, where $\alpha_q\geqslant 0$ satisfies $Q\alpha_q^2=q(1+\alpha_{q})$. In addition, if $q=0$, then
			\begin{equation*}
				\frac{\sqrt{\gamma_0}}{\sqrt{\gamma_0} \, k + \sqrt{Q}}
				\leqslant 
				\alpha_k\leqslant
				\frac{2\sqrt{Q}\alpha_0}{\sqrt{\gamma_0} \, k + 2\sqrt{Q}}.
			\end{equation*}		
			\item 
			If $q = 0$, then for all $k\geqslant 1$,
			\begin{equation*}
				\frac{Q}{(\sqrt{\gamma_0} \, k + \sqrt{Q})^2}
				\leqslant 
				\prod_{i=0}^{k-1}\frac{1}{1+\alpha_i}
				\leqslant
				\frac{4Q}{(\sqrt{\gamma_0} \, k + 2\sqrt{Q})^2}.
			\end{equation*}
		\end{itemize}
	\end{lem}
	\bibliographystyle{tfnlm}

\end{document}